\documentclass[reqno,11pt]{amsart}

\usepackage{amsmath}
\usepackage{amsfonts}
\usepackage{amssymb}
\usepackage{color}
\usepackage{mathrsfs}
\usepackage[colorlinks=true]{hyperref}
\usepackage{cleveref}

\newtheorem{thm}{Theorem}[section]
\newtheorem{lem}[thm]{Lemma}
\newtheorem{cor}[thm]{Corollary}
\newtheorem{prop}[thm]{Proposition}

\newtheorem{rem}[thm]{Remark}

\numberwithin{equation}{section}

\newcommand{\aequation}{\renewcommand{\theequation}{\mbox{A.\arabic{equation}}}}

\newcommand{\bequation}{\renewcommand{\theequation}{\mbox{B.\arabic{equation}}}}

\newcommand{\nequation}{\setcounter{equation}{0}}

\newcommand{\Wqb}{W_{q,\mathcal{B}}}
\newcommand{\R}{\mathbb{R}}
\newcommand{\C}{\mathbb{C}}

\newcommand{\mS}{\mathbb{S}}

\newcommand{\N}{\mathbb{N}}
\newcommand{\G}{\mathbb{G}}
\newcommand{\A}{\mathbb{A}}
\newcommand{\B}{\mathbb{B}}
\newcommand{\T}{\mathbb{T}}

\newcommand{\E}{\mathbb{E}}

\newcommand{\ml}{\mathcal{L}}

\newcommand{\Om}{\Omega}

\newcommand{\ve}{\varepsilon}
\newcommand{\rd}{\mathrm{d}}
\newcommand{\dom}{\mathrm{dom}}

\newcommand{\bear}{\begin{eqnarray}} 
\newcommand{\eear}{\end{eqnarray}} 
\newcommand{\bean}{\begin{eqnarray*}} 
\newcommand{\eean}{\end{eqnarray*}} 
\newcommand{\bs}{\begin{split}}
\newcommand{\es}{\end{split}}

\newcommand{\dhr}{\mathrel{\lhook\joinrel\relbar\kern-.8ex\joinrel\lhook\joinrel\rightarrow}}

\begin{document}

%\title[Stability in Age-Structured Diffusive Populations]{Stability Analysis for Equilibria in the Dynamics of Age-Structured Diffusive Populations}

\title[Stability in Age-Structured Diffusive Populations]{Stability and Instability of Equilibria in  Age-Structured Diffusive Populations}

\author{Christoph Walker}
\email{walker@ifam.uni-hannover.de}
\address{Leibniz Universit\"at Hannover\\ Institut f\" ur Angewandte Mathematik \\ Welfengarten 1 \\ D--30167 Hannover\\ Germany}

\date{\today}

\begin{abstract}
The principle of linearized stability and instability is established for a classical model describing the spatial movement of an age-structured population with nonlinear vital rates. It is shown that the real parts of the eigenvalues of the corresponding linearization at an equilibrium determine the latter's stability or instability. The key ingredient of the proof is the eventual compactness of the semigroup associated with the linearized problem, which is derived by a perturbation argument. The results are illustrated with examples.
\end{abstract}

%47D06 : One-parameter semigroups and linear evolution equations/ 47A10 : Spectrum, resolvent/  35B35 : Stability in context of PDEs / 35M10 : PDEs of mixed type/ 92D25 : Population dynamics (general)

\maketitle

%{\bf Christoph Walker}\footnote{ Christoph Walker\\ walker@ifam.uni-hannover.de \\ \\ Leibniz Universit\"at Hannover\\ Institut f\" ur Angewandte Mathematik \\ Welfengarten 1 \\ D--30167 Hannover, Germany }

%%%%%%%%%%%%%%%%%%%%%%%%%%%%%%%%%%%
\section{{\bf Introduction}}
%%%%%%%%%%%%%%%%%%%%%%%%%%%%%%%%%%%

The dynamics of a population structured by space and age is described by a density function $u=u(t,a,x)\ge 0$, where $t\ge 0$ refers to time, $a\in J:=[0,a_m]$ is the age variable with maximal age $a_m\in (0,\infty)$ (individuals may attain age greater than $a_m$ but are no longer tracked in the model), and $x\in \Omega$ is the spatial position within a domain $\Omega \subset \R^n$. 
Then
$$
\bar u(t,x):=\int_0^{a_m}\varrho(a,x) \, u(t,a,x)\, \rd a
$$
is the weighted local  overall population with weight $\varrho$ (i.e. the total number of individuals at time instant $t$ and spatial position $x$ when $\varrho\equiv 1$).
Assuming that the death rate $m=m(\bar u(t,x),a)\ge 0$ and the birth rate $b=b(\bar u(t,x),a)\ge 0$ depend on this quantity and on age,
the governing equations for the density $u$ are
\begin{subequations}\label{Eu1a}
\begin{align}
\partial_t u+\partial_a u&=\mathrm{div}_x\big(d(a,x)\nabla_xu\big)-m(\bar u(t,x),a)u\ , && t>0\, , &  a\in (0,a_m)\, ,& & x\in\Om\, ,\label{u1a}\\
u(t,0,x)&=\int_0^{a_m} b(\bar u(t,x),a)u(t,a,x)\,\rd a\, ,& & t>0\, , & & & x\in\Om\, ,\label{u2a}\\
\mathcal{B} u(t,a,x)&=0\ ,& & t>0\, , &  a\in (0,a_m)\, ,& & x\in\partial\Om\, ,\label{u3a}\\
u(0,a,x)&=u_0(a,x)\ ,& & &  a\in (0,a_m)\, , & & x\in\Om\,,\label{u4a}
\end{align}
\end{subequations}
where 
$$
\mathcal{B} u:=(1-\delta)u+\delta \partial_\nu u\,,\quad \delta\in\{0,1\}\,,
$$ 
means either Dirichlet boundary conditions $u\vert_{\partial\Omega}=0$ if $\delta=0$ or Neumann boundary conditions~$\partial_\nu u=0$ if $\delta=1$. Note that the evolution problem~\eqref{Eu1a} exhibits hyperbolic (due to the aging term) and parabolic (due to the diffusion) features and involves a nonlocal condition~\eqref{u2a} with respect to age.

Since many years linear and nonlinear age-structured populations with spatial diffusion have been the focus of intensive research, see e.g. \cite{WebbSpringer,WalkerJEPE} and the references therein. In particular, problems of the form~\eqref{Eu1a} \cite{DelgadoMolinaSuarez05,Rhandi98,RhandiSchnaubelt_DCDS99,WalkerDCDSA10} or variants thereof such as models including nonlocal diffusion~\cite{KangRuanJMB21} or compartmental models for infectious diseases spreading \cite{ChekrounKuniya19,ChekrounKuniya20a,ChekrounKuniya20b,DucrotMagal10,DucrotMagal11,KuniyaOizumi15} have been studied  by various authors addressing questions related e.g. to well-posedness or qualitative aspects under different assumptions (none of these reference lists is close to being complete though). The present paper contributes to the study of stability of equilibria to ~\eqref{Eu1a}. While most research so far on stability of equilibria in age-structured diffusive populations apply the principle of linearized stability in an ad-hoc fashion, the aim of the present paper is to provide a proof therefor.

The existence of (nontrivial) equilibria (i.e. time-in\-de\-pen\-dent solutions) to~\eqref{Eu1a} has been established under fairly general conditions by the author in a series of papers using fixed point methods~\cite{WalkerJDE10} or bifurcations techniques~\cite{WalkerSIMA09,WalkerJDE10,WalkerJDDE13}. %The aim herein is to derive stability properties of an equilibrium from the linearization of \eqref{Eu1a} at such an equilibrium.
A principle of linearized stability for age-structured populations without spatial diffusion was established in~\cite{Pruess83}, see also~\cite{WebbBook}. As for the case including spatial diffusion a criterion for linearized stability was derived in a recent paper~\cite{WalkerZehetbauerJDE}. Herein, we shall refine and simplify considerably this stability result and complement it with an instability result. In particular, we show that the spectrum of the linearization (as an unbounded operator) indeed consists  of eigenvalues only whose real parts determine stability and instability. \\

To give a first flavor of our findings we present a paraphrased version for the particular case of the trivial equilibrium $\phi=0$. In the next section we will state a more general version for an arbitrary equilibrium.

Writing  the unique strong solution $v$ to the heat equation
\begin{align*}
\partial_a v=\mathrm{div}_x\big(d(a,x)\nabla_xv\big)\,,\quad (a,x)\in J\times \Omega\,,\qquad  
v(0,x)=v_0(x)\,,\quad x\in\Omega\,,
\end{align*}
subject to the boundary condition $\mathcal{B} v=0$ on $\partial\Omega$ and the initial value $v_0\in L_q(\Omega)$ in the form $v(a)=\Pi_*(a,0)v_0$, $a\in J$, we define by
\begin{align*}
Q_0:=
\int_0^{a_m} b(0,a)\, \exp\left(-\int_0^a m(0,\sigma)\,\rd \sigma\right)\,\Pi_{*}(a,0)\,\rd a
\end{align*}
a (compact and positive) operator on $L_q(\Omega)$ and denote by  $r(Q_{0})$ its spectral radius. Then the stability property  of the trivial equilibrium is determined according to:

%%%%%%%%%%%%%%%%%%%%%%%%%%
%%%%%%%%%%%%%%%%%%%%%%%%%%
\begin{prop}\label{P21xx}
Let $q>n$ and assume that 
$$
d:J\times\bar\Omega\to (0,\infty)\,,\qquad 
b: \R\times J \to (0,\infty)\,,\qquad m: \R\times J \to \R^+ 
$$
are (sufficiently) smooth functions. 
\begin{itemize}
\item[\bf (a)] If $r(Q_{0})<1$, then the trivial equilibrium to~\eqref{Eu1a} is exponentially asymptotically stable in~$L_1\big((0,a_m),W_{q}^{1}(\Omega)\big)$.

\item[\bf (b)] If $r(Q_{0})>1$, then the trivial equilibrium to~\eqref{Eu1a} is unstable in~$L_1\big((0,a_m),W_{q}^{1}(\Omega)\big)$.
\end{itemize}
In case of Neumann boundary conditions (i.e. $\delta=1$), the spectral radius is
$$
r(Q_{0})=\int_0^{a_m} b(0,a)\, \exp\left(-\int_0^a m(0,s)\rd s\right)\,\rd a\,.
$$
\end{prop} 
%%%%%%%%%%%%%%%%%%%%%%%%%%
%%%%%%%%%%%%%%%%%%%%%%%%%%

Proposition~\ref{P21xx} is a special case of Proposition~\ref{P21} below.
In fact, we can prove a much more general result for an arbitrary equilibrium. For this purpose, we shall consider problem~\eqref{Eu1a} in an abstract setting and introduce the notation\footnote{We suppress the $x$-variable consistently in the abstract formulation by considering  functions with values in function spaces on $\Omega$. In particular, the data $m$, $b$, $\varrho$ may  include an $x$-dependence and $u(t,a)\in E_0$.}
$$
A(a)w:=\mathrm{div}_x\big(d(a,\cdot)\nabla_xw\big)\, ,\quad w\in E_1\, ,
$$
where e.g. $E_1:=\Wqb^2(\Om)$ with $q\in (1,\infty)$ denotes the Sobolev space of functions \mbox{$w\in W_q^{2}(\Om)$}  satisfying the boundary condition $\mathcal{B} w=0$ on $\partial\Om$. Then $A(a)$ is for each $a\in J$ the generator of an analytic semigroup on the Banach lattice $E_0:=L_q(\Omega)$ with compactly and densely embedded domain $E_1$. %That is, $A$ induces an evolution operator on $E_0$ with regularity subspace $E_1$ (see Appendix~\ref{App:EvolSys}). 
The abstract formulation of~\eqref{Eu1a} now reads
\begin{subequations}\label{PPPP} 
\begin{align}
\partial_t u+ \partial_au \, &=     A(a)u -m(\bar u(t),a)u \,, \qquad t>0\, ,\quad a\in (0,a_m)\, ,\label{P1}\\ 
u(t,0)&=\int_0^{a_m}b(\bar u(t),a)\, u(t,a)\, \rd a\,, \qquad t>0\, ,\label{P2} \\
u(0,a)&=  u_0(a)\,, \qquad a\in (0,a_m)\,,
\end{align}
\end{subequations}
with
$$
\bar u(t):=\int_0^{a_m}\varrho(a) \, u(t,a)\, \rd a\,.
$$
We then shall focus on~\eqref{PPPP} and present our main stability result for this problem. Later we interpret our findings for the concrete equation~\eqref{Eu1a} and variants thereof. The regularizing effects from the diffusion, reflected in~\eqref{PPPP} by the operator $A$, are of great importance since they will allow us to handle the nonlinearities under mild assumptions (mainly on the regularity of the vital rates $m$ and $b$).

%%%%%%%%%%%%%%%%%%%%%%%%%%%%%%%%%%%%%%%%%%%%%%%%%%%%
%%%%%%%%%%%%%%%%%%%%%%%%%%%%%%%%%%%%%%%%%%%%%%%%%%%%
%%%%%%%%%%%%%%%%%%%%%%%%%%%%%%%%%%%%%%%%%%%%%%%%%%%%
\section{Main Results}\label{Sec2}
%%%%%%%%%%%%%%%%%%%%%%%%%%%%%%%%%%%%%%%%%%%%%%%%%%%%
%%%%%%%%%%%%%%%%%%%%%%%%%%%%%%%%%%%%%%%%%%%%%%%%%%%%
%%%%%%%%%%%%%%%%%%%%%%%%%%%%%%%%%%%%%%%%%%%%%%%%%%%%

%%%%%%%%%%%%%%%%%%%%%%%%%%%%%%%%%%%%%%%%%%%%%%%%%%%%
%%%%%%%%%%%%%%%%%%%%%%%%%%%%%%%%%%%%%%%%%%%%%%%%%%%%
%%%%%%%%%%%%%%%%%%%%%%%%%%%%%%%%%%%%%%%%%%%%%%%%%%%%
\subsection*{General Assumptions and Notations}
%%%%%%%%%%%%%%%%%%%%%%%%%%%%%%%%%%%%%%%%%%%%%%%%%%%%
%%%%%%%%%%%%%%%%%%%%%%%%%%%%%%%%%%%%%%%%%%%%%%%%%%%%
%%%%%%%%%%%%%%%%%%%%%%%%%%%%%%%%%%%%%%%%%%%%%%%%%%%%

Set $J:=[0,a_m]$. Throughout the following,~$E_0$ is a real Banach lattice ordered by a closed convex cone $E_0^+$  (in the following we do not distinguish $E_0$ from its complexification required at certain points) and 
\begin{equation*}%\label{A0}
E_1\stackrel{d}{\dhr} E_0\,,
\end{equation*}
that is, $E_1$ is a densely and compactly embedded subspace of $E_0$. We write $\ml(E_1,E_0)$ for the Banach space of bounded linear operators from $E_1$ to $E_0$, set \mbox{$\ml(E_0):=\ml(E_0,E_0)$}, and denote by~$\ml_+(E_0)$  the positive operators. For a (possibly unbounded) operator $$\mathcal{A}:\mathrm{dom}(\mathcal{A})\subset E_0\to E_0$$ we mean by $D(\mathcal{A})$ its domain $\mathrm{dom}(\mathcal{A})$ endowed with the graph norm. For $\theta\in (0,1)$ and an admissible interpolation functor $(\cdot,\cdot)_\theta$ (see~\cite{LQPP}), we put $E_\theta:= (E_0,E_1)_\theta$ and  equip it with the order naturally induced by $E_0^+$. We use the notion $$\E_\theta:=L_1(J,E_\theta)\,,\quad \theta\in [0,1]\,,$$
and observe $\E_\theta\hookrightarrow \E_0$ for $\theta\in [0,1]$.
We assume that there is $\rho>0$ such that
\begin{subequations}\label{A1}
\begin{equation}
A\in  C^\rho\big(J,\mathcal{H}(E_1,E_0)\big)
\end{equation}
and
\begin{equation}\label{A1aa}
	A(a) \ \text{ is resolvent positive for each $a\in J$}\,,
\end{equation}
\end{subequations}
where $\mathcal{H}(E_1,E_0)$ is the subspace of $\ml(E_1,E_0)$ consisting of all generators of analytic semigroups on $E_0$ with domain $E_1$. 
Then \eqref{A1} and \cite[II.Corollary 4.4.2]{LQPP} imply that $A$ generates a positive, parabolic evolution operator 
$$
\big\{\Pi(a,\sigma)\in\ml(E_0)\,;\, a\in J\,,\, 0\le\sigma\le a\big\}\,,
$$ 
on $E_0$ with regularity subspace $E_1$ in the sense of \cite[Section~II.2.1]{LQPP} (see Appendix~\ref{App:EvolSys} for a summary of the most important properties of parabolic evolution operators). 

%(if $a_m<\infty$ this follows directly from~\cite[II.Lemma 5.1.3]{LQPP}) and%
%	\begin{equation}\label{A4}%
%	\text{if $a_m=\infty$, then $\varpi<0$ in \eqref{EO}}\,.
%	\end{equation}

%%%%%%%%%%%%%%%%%%%%%%%%%%%%%%%%%%%%%%%%%%%%%%%%%%%%
%%%%%%%%%%%%%%%%%%%%%%%%%%%%%%%%%%%%%%%%%%%%%%%%%%%%
%%%%%%%%%%%%%%%%%%%%%%%%%%%%%%%%%%%%%%%%%%%%%%%%%%%%

\subsection*{Well-Posedness} 

%%%%%%%%%%%%%%%%%%%%%%%%%%%%%%%%%%%%%%%%%%%%%%%%%%%%
%%%%%%%%%%%%%%%%%%%%%%%%%%%%%%%%%%%%%%%%%%%%%%%%%%%%
%%%%%%%%%%%%%%%%%%%%%%%%%%%%%%%%%%%%%%%%%%%%%%%%%%%%

Before stating our main stability result, let us recall the well-posedness of the nonlinear problem~\eqref{PPPP} established in \cite{WalkerDCDSA10,WalkerZehetbauerJDE}. In the following, let $\alpha\in[0,1)$ be fixed. We assume for the birth and the death rate that
\begin{subequations}\label{A11}
\begin{align}\label{A1b}
\big[\bar v\to b(\bar v,\cdot)]&\in C_b^{1-}\big( E_\alpha, L_{\infty}^+\big(J,\ml(E_\alpha,E_0)\big)\big)\,,\\
	\big[\bar v\to m(\bar v,\cdot)]&\in C_b^{1-}\big( E_\alpha, L_{\infty}^+\big(J,\ml(E_\alpha,E_0)\big)\big)\,,\label{A1c}
\end{align}
where $C_b^{1-}$ stands for locally Lipschitz continuous maps that are bounded on bounded sets. 
The weight function $\varrho$ is such that  
\begin{equation}\label{A1d}
	\varrho\in C\big(J,\ml_+(E_\theta)\big)\,,\quad \theta\in\{0,\alpha,\vartheta\}\,,
\end{equation}
\end{subequations}
for some $\vartheta\in (0,1)$ (if $\alpha\in (0,1)$, then it suffices to take $\vartheta=\alpha$). We   use the notation
$$
\bar v:=\int_0^{a_m} \varrho(a)\,v(a)\,\rd a\in E_\theta\,,\quad v\in \E_\theta\,.
$$

Observe that integrating \eqref{PPPP} formally along characteristics  yields the necessary condition that a solution $u:\R^+\rightarrow\E_0$ 
 with initial value $u_0\in\E_0$ satisfies the fixed point equation
\begin{subequations}\label{1000}
  \begin{equation}\label{u}
     u(t,a)\, =\, \left\{ \begin{aligned}
    &\Pi(a,a-t)\, u_0(a-t) + G_{F(u)}(t,a)\, ,& &   a\in J\,,\ 0\le t\le a\, ,\\
    & \Pi(a,0)\, B_u(t-a)+ G_{F(u)}(t,a)\, ,& &  a\in J\, ,\ t>a\, ,
    \end{aligned}
   \right.
    \end{equation}
where $F(u):=-m(\bar u,\cdot)u$ and
\begin{equation}\label{4000}
G_{v}(t,a):=\int_{(t-a)_+}^{t}\Pi(a,a-t+s)\, v(s,a-t+s)\,\rd s
 \end{equation}
for $v:\R^+\rightarrow \E_\alpha$,
and where $B_u$ satisfies the Volterra equation 
    \begin{equation}\label{5000}
    \begin{split}
    B_u(t)\, & =\, \int_0^t  b(\bar u(t),a) \Pi(a,0) B_u(t-a)\, \rd
    a + \int_t^{a_m} b(\bar u(t),a)\Pi(a,a-t)\, u_0(a-t)\, \rd a\\
    &\quad +\int_0^{a_m} b(\bar u(t),a) G_{F(u)}(t,a)\,\rd a
	\end{split}
    \end{equation}		
for $t\ge 0$ (we set $b(\bar v,a):=0$ whenever $a\notin J$). That is, $u(t,0)=B_u(t)$ for $t\ge 0$ by ~\eqref{u}, while~\eqref{5000} implies
\begin{equation*}\label{5000b}
	\begin{split}
		B_u(t)\, & =\, \int_0^{a_m}  b(\bar u(t),a)\, u(t,a)\, \rd
		a\,,\quad t\ge 0\,.
	\end{split}
\end{equation*}
\end{subequations}
The following result was established in \cite{WalkerZehetbauerJDE} (see also \cite{WalkerDCDSA10}):

\begin{prop}\label{T1JDE}
Suppose \eqref{A1} and \eqref{A11}. 
For every $u_0\in \E_\alpha$ there exists a unique maximal solution 
$u=u(\cdot;u_0)\in C\big(I(u_0),\E_\alpha\big)$
to problem~\eqref{PPPP} on some maximal interval of existence $I(u_0)=[0,T_{max}(u_0))$; that is, $u(t;u_0)$ satisfies~\eqref{1000} for $t\in I(u_0)$. 
 If 
$$
\sup_{t\in I(u_0)\cap [0,T]} \| u(t;u_0)\|_{\E_\alpha} <\infty
$$
for every $T>0$, then the solution exists globally, i.e., $I(u_0)=\R^+$.
Finally, if $u_0\in\E_\alpha^+$, then  $u(t;u_0)\in\E_\alpha^+$ for~$t\in I(u_0)$.
\end{prop} 

\begin{proof}
This is \cite[Proposition 2.1]{WalkerZehetbauerJDE}.
\end{proof}

Note that assumptions~\eqref{A1} and~\eqref{A11} are not really restrictive and satisfied for (sufficiently) smooth functions~$m, b, \varrho$ and diffusion operators as in the introduction, see Section~\ref{Sec5}.

%%%%%%%%%%%%%%%%%%%%%%%%%%%%%%%%%%%%%%%%%%%%%%%%%%%%
%%%%%%%%%%%%%%%%%%%%%%%%%%%%%%%%%%%%%%%%%%%%%%%%%%%%
%%%%%%%%%%%%%%%%%%%%%%%%%%%%%%%%%%%%%%%%%%%%%%%%%%%%

\subsection*{Linearized Stability and Instability}

In the following, an equilibrium (i.e. a time-independent solution) $\phi\in C(J,E_\alpha)$ to~\eqref{PPPP} is a mild solution (see \eqref{VdKx}) to 
	\begin{align}\label{EP} 
		\partial_a\phi \, &=     A(a)\phi -m(\bar \phi,a)\phi \,, \quad a\in (0,a_m)\, ,\qquad
		\phi(0)=\int_0^{a_m}b(\bar \phi ,a)\, \phi(a)\, \rd a\,. 
	\end{align}
Clearly, $\phi\equiv 0$ is always an equilibrium. As pointed out above,  fairly general conditions sufficient for the existence of at least one positive non-trivial equilibrium $\phi\in\E_1 \cap C(J,E_\alpha)$ were presented in earlier works~\cite{WalkerSIMA09,WalkerJDE10,WalkerJDDE13}.

An equilibrium $\phi\in C(J,E_\alpha)$ to~\eqref{PPPP} is said to be {\it stable} in $\E_\alpha$ provided that for every $\ve>0$ there exists $\delta>0$ such that, if $u_0\in \mathbb{B}_{\E_\alpha}(\phi,\delta)$, then $T_{max}(u_0)=\infty$ and $u(t;u_0)\in \mathbb{B}_{\E_\alpha}(\phi,\ve)$ for every $t\ge 0$, where $u(\cdot,u_0)$ denotes the maximal solution to~\eqref{PPPP} from Proposition~\ref{T1JDE}. The equilibrium $\phi\in C(J,E_\alpha)$ is {\it asymptotically exponentially stable} in $\E_\alpha$, if it is stable and there are $r>0$ and $M>0$ such that 
$$
\|u(t;u_0)-\phi\|_{\E_\alpha}\le M e^{-rt}\|u_0-\phi\|_{\E_\alpha}\,,\quad t\ge 0\,,
$$ 
for $u_0\in \mathbb{B}_{\E_\alpha}(\phi,\delta)$. Finally,  an equilibrium $\phi\in C(J,E_\alpha)$ is {\it unstable} in $\E_\alpha$, if it is not stable.

%%%%%%%%%%%%%%%%%%%%%%%%%%%%%%%%%%%
%%%%%%%%%%%%%%%%%%%%%%%%%%%%%%%%%%%
\subsubsection*{Assumptions}
%%%%%%%%%%%%%%%%%%%%%%%%%%%%%%%%%%%
%%%%%%%%%%%%%%%%%%%%%%%%%%%%%%%%%%%
Let $\phi\in\E_1 \cap C(J,E_\alpha)$ be a fixed equilibrium  to~\eqref{PPPP}. We assume that the death and the birth rate are continuously Fr\'echet differentiable at $\bar\phi$. More precisely, for  $\alpha\in [0,1)$ still fixed, we assume that
 \begin{subequations}\label{B}
\begin{align}
	& E_\alpha\rightarrow L_{\infty}\big(J,\ml(E_\alpha,E_0)\big)\,,\ z\mapsto m(z,\cdot) \text{ is differentiable at $\bar\phi\in E_1$}\,,\label{B1}\\
		& E_\alpha\rightarrow L_{\infty}\big(J,\ml(E_\alpha,E_0)\big)\,,\ z\mapsto b(z,\cdot) \text{ is differentiable at $\bar\phi\in E_1$}\,,\label{B2}
	\end{align}
such that for  $v\in \E_0$ we can write (with $\partial$ indicating Fr\'echet derivatives with respect to~$\bar\phi$)
\begin{equation}\label{14Ga}
			m(\bar v,\cdot)v-m(\bar \phi,\cdot)\phi=m(\bar\phi,\cdot)( v- \phi) +\partial m(\bar\phi,\cdot)[\bar v-\bar \phi]\phi+ R_m( v- \phi)
		\end{equation}
and
\begin{equation}\label{14Ea}
			b(\bar v,\cdot)v-b(\bar \phi,\cdot)\phi=b(\bar\phi,\cdot)( v- \phi) +\partial b(\bar\phi,\cdot)[\bar v-\bar \phi]\phi+ R_b( v- \phi)\,,
		\end{equation}
where for the reminder terms $R_m:\E_\alpha\to \E_0$ and $R_b:\E_\alpha\to \E_0$ there exists an increasing function~$d_o\in C(\R^+,\R^+)$ with $d_o(0)=0$ such that $d_o(r)>0$ for each $r>0$ with
	\begin{equation}\label{14G2}
		\|R_m(v)\|_{\E_0}+	\|R_b( v)\|_{\E_0}\le d_o(r)\, \| v\|_{\E_\alpha}\,,\quad \|v\|_{\E_\alpha}\le r\,,
		\end{equation}
and
\begin{equation}\label{do2}
\|R_m(v_1)-R_m(v_2)\|_{\E_0} +\|R_b(v_1)-R_b(v_2)\|_{\E_0}\le d_o(r)\,\|v_1-v_2\|_{\E_\alpha}\,,\quad \|v_1\|_{\E_\alpha}, \|v_2\|_{\E_\alpha}\le r\,.
\end{equation}
For technical reasons,  we assume for the birth rate that (for some $\vartheta\in (0,1)$, see~\eqref{A1d})
\begin{equation}\label{B3}
	\begin{split}
		b(\bar \phi,\cdot)\in   C\big(J,\ml(E_0)\big)\cap L_{\infty}\big(J,\ml(E_\theta)\big)\,,\quad \theta\in\{0,\alpha,\vartheta\}\,,
	\end{split}
\end{equation}
and 
\begin{equation}\label{B2b}
\big[z\mapsto \partial b(\bar \phi,\cdot)[z]\phi\big]\in \ml\big(E_\theta, \E_\theta\big)\,,\quad \theta\in\{0,\alpha,\vartheta\}\,,
\end{equation}
while for the death rate we impose that
\begin{equation}\label{B2bb}
m(\bar \phi,\cdot)\in   C^\rho\big(J,\ml(E_\beta,E_0)\big)
\end{equation}
for some $\beta\in [0,1)$ and
\begin{equation}\label{B2bbb}
\big[a\mapsto \partial m(\bar \phi,a)[\cdot]\phi(a)\big]\in C\big(J,\ml(E_0)\big)\,.
\end{equation} 
\end{subequations}
The fact that we can handle nonlinearities $m$ and $b$ being defined on interpolation spaces~$E_\alpha$  guarantees great flexibility in concrete applications. Indeed, the assumptions imposed above are rather easily checked in  problems such as~\eqref{Eu1a}  since they are mainly assumptions on the regularity of the data (see Section~\ref{Sec5} for details).  \\

In \cite{WalkerZehetbauerJDE} it was shown that the stability of an equilibrium $\phi$ can be deduced from the (formal) linearization of~\eqref{PPPP} at~$\phi$ given by
\begin{subequations}\label{EPx} 
\begin{align} 
	\partial_t v+ \partial_av \, &=     A(a)v -m\big(\bar \phi,a\big) v -\partial m\big(\bar \phi,a\big)[\bar v(t)]\phi(a)\,, \qquad t>0\, ,\quad a\in (0,a_m)\, ,\label{e12}\\ 
	v(t,0)&=\int_0^{a_m}b\big(\bar \phi,a\big)\, v(t,a)\, \rd a +\int_0^{a_m}\partial b\big(\bar \phi,a\big)[\bar v(t)]\, \phi(a)\, \rd a \,, \qquad t>0\, , \label{e12b}  \\
	v(0,a)&=  v_0(a)\,, \qquad a\in (0,a_m)\,,
\end{align}
\end{subequations}  
with $\partial$ indicating Fr\'echet derivatives with respect to $\bar\phi$. More precisely, according to~\cite{WalkerZehetbauerJDE}, an equilibrium~$\phi$ is locally asymptotically stable if the semigroup associated with the linearization~\eqref{EPx} has a negative growth bound. The characterization of the latter, however, was left open. The aim of the present research now is to refine and improve this stability result and complement it with an instability result. Concretely, we prove that the real parts of the eigenvalues of the generator of the semigroup associated with the linearization~\eqref{EPx} determine stability or instability of the equilibrium. Thus, we establish the classical principle of linearized stability for~\eqref{PPPP}. 

In the following, an {\it eigenvalue} of the generator associated with~\eqref{EPx} means a number $\lambda\in\C$ for which there is a nontrivial mild solution $w\in C(J,E_0)$, $w\not\equiv 0$, to
\begin{subequations}\label{eigenvalueproblem} 
\begin{align} 
	\lambda w+ \partial_aw \, &=     A(a)w -m(\bar \phi,a) w -\partial m(\bar \phi,a)[\bar w]\phi(a)\,, \qquad a\in (0,a_m)\, ,\label{k1}\\ 
	w(0)&=\int_0^{a_m}b(\bar \phi,a)\, w(a)\, \rd a +\int_0^{a_m}\partial b(\bar \phi,a)[\bar w]\, \phi(a)\, \rd a \,.
\end{align}
\end{subequations}
Here is the main result:\vspace{2mm}

\begin{thm}~\label{MainT}
Let $\alpha\in [0,1)$. Assume \eqref{A1}, \eqref{A11}, and \eqref{B}, where $\phi\in\E_1 \cap C(J,E_\alpha)$ is an equilibrium to~\eqref{PPPP}.  The following hold:
\begin{itemize}
\item[\bf (a)] If $\mathrm{Re}\,\lambda < 0$ for any eigenvalue $\lambda$ to~\eqref{eigenvalueproblem},   then $\phi$  is exponentially asymptotically stable in~$\E_\alpha$.

\item[\bf (b)] If $\mathrm{Re}\,\lambda > 0$ for some eigenvalue $\lambda$ to~\eqref{eigenvalueproblem},  then $\phi$  is unstable in $\E_\alpha$.
\end{itemize}
\end{thm}

Theorem~\ref{MainT} follows from Theorem~\ref{TInStable} below and is an extension of \cite[Theorem~2]{Pruess83}, \cite[Theorem~4.13]{WebbBook} to the case with spatial diffusion. We point out again that the ability to work in the spaces $\E_\alpha$ (instead of only in $\E_0$) allows us to treat general nonlinearities in the vital rates $m$ and $b$, see Section~\ref{Sec5}. 

Statement~{\bf (a)} of Theorem~\ref{MainT} refines the stability result of~\cite[Theorem~2.2]{WalkerZehetbauerJDE}. Indeed, one of our main achievements herein is the characterization of the growth bound of the semigroup associated with~\eqref{EPx} in terms of the spectral bound of the semigroup generator. Moreover, with the instability statement~{\bf (b)} we complete~\cite[Theorem~2.2]{WalkerZehetbauerJDE} and so provide a concise characterization of stability or instability of an equilibrium to~\eqref{PPPP} via the linear eigenvalue problem~\eqref{eigenvalueproblem}.
Actually, we shall later see in Section~\ref{SecRephrase} that the latter can be rephrased in a somewhat more accessible way for applications.

The crucial key for the proof of Theorem~\ref{MainT} is to show that the semigroup on $\E_0$ associated with the linear problem~\eqref{EPx} (as well as its restriction to~$\E_\alpha$) is eventually compact. As a consequence the growth bound of the semigroup and the spectral bound of the corresponding generator coincide and further fundamental spectral properties of the generator can be derived. We emphasize that, even though the semigroup associated with~\eqref{EPx} when $m\equiv 0$ is known to have this eventual compactness property (see~\cite{WalkerIUMJ}), it is by no means obvious that this property is inherited when including a nontrivial death rate~$m$. This is due to the fact that, on the one hand, a perturbation of the generator of an eventually compact semigroup does not in general generate again an eventually compact semigroup, and, on the other hand, that the  perturbation $\partial m(\bar \phi,a)[\bar w]\phi(a)$ appearing in~\eqref{e12} constitutes a {\it nonlocal} perturbation with respect to $w$. In order to establish the fundamental compactness property nonetheless, we make use of the particular form of this perturbation, see~\eqref{qqq}.

\subsection*{Paper Outline} In Section~\ref{Sec3} we first focus on the linear problem~\eqref{EPx} and prove the eventual compactness of the corresponding semigroup on $\E_0$ and of its restriction to~$\E_\alpha$.  This and the implied spectral properties of the generator are summarized in Theorem~\ref{T:NormCont} and Corollary~\ref{C:Ealpha}. 

In Section~\ref{Sec4} we use the crucial fact derived in \cite{WalkerZehetbauerJDE} that the difference $u(\cdot;u_0)-\phi$ can be represented in terms of the linearization semigroup associated with~\eqref{EPx}, see~\eqref{20}. Since the growth bound of the linearization semigroup is determined by the spectral bound of the semigroup generator as shown previously in Section~\ref{Sec3}, this yields statement~{\bf (a)} of Theorem~\ref{MainT}. Moreover, the construction of backwards solutions to problem~\eqref{PPPP} in Lemma~\ref{backwards} under the assumptions of statement~{\bf (b)} of Theorem~\ref{MainT} then leads to the instability result. This part of the proof is inspired by~\cite[Theorem~4.13]{WebbBook}. In Proposition~\ref{P47} we give an alternative formulation of the eigenvalue problem~\eqref{eigenvalueproblem}.

In Section~\ref{Sec5} we consider concrete examples and prove, in particular, Proposition~\ref{P21xx} on the stability analysis of the trivial equilibrium to problem~\eqref{Eu1a}. Moreover, we provide an instability result for a nontrivial equilibrium, see Proposition~\ref{P50}.

Finally, two appendices are included. In Appendix~\ref{App:NormCont} we provide the technical and thus postponed proof of Proposition~\ref{Prop:NormCont}. In Appendix~\ref{App:EvolSys} we briefly summarize the main properties of parabolic evolution operators which play an important role in our analysis and are used throughout.

%%%%%%%%%%%%%%%%%%%%%%%%%%%%%%%%%%%%%%%%%%%%%%%%%%%%
%%%%%%%%%%%%%%%%%%%%%%%%%%%%%%%%%%%%%%%%%%%%%%%%%%%%
%%%%%%%%%%%%%%%%%%%%%%%%%%%%%%%%%%%%%%%%%%%%%%%%%%%%

%%%%%%%%%%%%%%%%%%%%%%%%%%%%%%%%%%%%%%%%%%%%%%%%%
%%%%%%%%%%%%%%%%%%%%%%%%%%%%%%%%%%%%%%%%%%%%%%%%%
%%%%%%%%%%%%%%%%%%%%%%%%%%%%%%%%%%%%%%%%%%%%%%%%%
\section{{\bf The Linear Problem}}\label{Sec3}
%%%%%%%%%%%%%%%%%%%%%%%%%%%%%%%%%%%%%%%%%%%%%%%%%
%%%%%%%%%%%%%%%%%%%%%%%%%%%%%%%%%%%%%%%%%%%%%%%%%
%%%%%%%%%%%%%%%%%%%%%%%%%%%%%%%%%%%%%%%%%%%%%%%%%

In order to prepare the proof of Theorem~\ref{MainT} we focus our attention first on the linear problem
\begin{subequations}\label{PP} 
\begin{align}
\partial_t u+ \partial_au \, &=     A_\ell(a)u \,, \qquad t>0\, ,\quad a\in (0,a_m)\, ,\label{1}\\ 
u(t,0)&=\int_0^{a_m}b_\ell(a)\, u(t,a)\, \rd a\,, \qquad t>0\, ,\label{2} \\
u(0,a)&=  \psi(a)\,, \qquad a\in (0,a_m)\,,
\end{align}
\end{subequations}
where $A_\ell$ satisfies
\begin{equation}\label{A1l}
A_\ell\in  C^\rho\big(J,\mathcal{H}(E_1,E_0)\big)
\end{equation}
for some $\rho>0$ and where
we impose  for the birth rate that there is $\vartheta\in (0,1)$ with
\begin{equation}\label{A2l}
b_\ell\in L_{\infty}\big(J,\ml(E_\theta)\big)\,, \quad \theta\in  \{0,\vartheta\}  
\,.
\end{equation}
We then denote by
$$
\big\{\Pi_\ell(a,\sigma)\in\ml(E_0)\,;\, a\in J\,,\, 0\le\sigma\le a\big\}
$$ 
the parabolic evolution operator 
on $E_0$ with regularity subspace $E_1$ generated by $A_\ell$ (see Appendix~\ref{App:EvolSys})
and use the notation
$$
\|b_\ell\|_\theta:=\|b_\ell\|_{L_{\infty}(J,\ml(E_\theta))}\,.
$$

%%%%%%%%%%%%%%%%%%%%%%%%%%%%%%%%%%%%%%%%%%%%%%%%%%%%
\subsection{Preliminaries}
%%%%%%%%%%%%%%%%%%%%%%%%%%%%%%%%%%%%%%%%%%%%%%%%%%%%

We first recall some of the results from~\cite{WalkerIUMJ}. Formal integrating of~\eqref{1} along characteristics  entails that the solution 
$$
[\mS(t)\psi](a):=u(t,a)\,,\qquad t\ge 0\,,\quad a\in J\,,
$$ 
to \eqref{PP}  for given $\psi\in\E_0=L_1(J,E_0)$ is of the form
\begin{subequations}\label{100}
  \begin{equation}
     \big[\mS(t) \psi\big](a)\, :=\, \left\{ \begin{aligned}
    &\Pi_\ell(a,a-t)\, \psi(a-t)\, ,& &   a\in J\,,\ 0\le t\le a\, ,\\
    & \Pi_\ell(a,0)\, \mathsf{B}_\psi(t-a)\, ,& &  a\in J\, ,\ t>a\, ,
    \end{aligned}
   \right.
    \end{equation}
with $\mathsf{B}_\psi:=u(\cdot,0)$ satisfying the linear Volterra equation 
    \begin{equation}\label{500}
\begin{split}
    \mathsf{B}_\psi(t)\, =\, &\int_0^t \chi(a)\, b_\ell(a)\, \Pi_\ell(a,0)\, \mathsf{B}_\psi(t-a)\, \rd
    a\, \\ & +\, \int_0^{a_m-t} \chi(a+t)\, b_\ell(a+t)\, \Pi_\ell(a+t,a)\, \psi(a)\, \rd a\, ,\quad
    t\ge 0\, ,
\end{split}
    \end{equation}
\end{subequations}		
where $\chi$ is the characteristic function of the interval $(0,a_m)$. Note that $\mathsf{B}_\psi$ is such that
  \begin{equation}\label{6a}
    \mathsf{B}_\psi(t)= \int_0^{a_m} b_\ell(a) \big[\mS(t)\psi\big](a)\, \rd a\, ,\quad t\ge 0\, .
    \end{equation}
It follows from \cite{WalkerIUMJ} that there is a unique solution 
 \begin{equation}\label{BBBB}
[\psi\mapsto \mathsf{B}_\psi]\in\ml \big(\E_0, C(\R^+,E_0)\big)
 \end{equation} 
to~\eqref{500} and that $(\mS(t))_{t\ge 0}$ defines a strongly continuous semigroup on $\E_0$ enjoying the property of eventual compactness and exhibiting regularizing effects induced by the parabolic evolution operator~$\Pi_\ell$. Moreover, its generator can be characterized fully. We summarize those properties which will be important for our purpose herein:

%%%%%%%%%%%%%%%%%%%%%%%%%%%%%%%%%%%%%%%%%%%%%%%%%%
%%%%%%%%%%%%%%%%%%%%%%%%%%%%%%%%%%%%%%%%%%%%%%%%%%
\begin{thm}\label{IUMJT1}
Suppose \eqref{A1l} and \eqref{A2l}. \vspace{2mm}

{\bf (a)} $(\mS(t))_{t\ge 0}$ defined in \eqref{100} is a strongly continuous, eventually compact semigroup on the space \mbox{$\E_0=L_1(J,E_0)$}.  

If $A_\ell(a)$ is resolvent positive for every $a\in J$ and if $b_\ell\in L_{\infty}\big(J,\ml_+(E_0)\big)$, then the semigroup~$(\mS(t))_{t\ge 0}$ is positive. \vspace{2mm}

{\bf (b)} In fact, given $\alpha\in [0,1)$, the restriction $(\mS(t)\vert_{\E_\alpha})_{t\ge 0}$   defines a strongly continuous  semigroup  on $\E_\alpha$, and there are $M_\alpha\ge 1$ and $\varkappa_\alpha\in\R$ such that 
\begin{equation}\label{E3}
\|\mS(t)\|_{\ml(\E_\alpha)}+t^{\alpha}\|\mS(t)\|_{\ml(\E_0,\E_\alpha)}\le M_\alpha\, e^{\varkappa_\alpha t} \,,\quad t\ge 0\,.
\end{equation}

{\bf (c)} Denote by $\A$ the infinitesimal generator of the semigroup~$(\mS(t))_{t\ge 0}$ on $\E_0$. Then
\mbox{$\psi\in \dom(\A)$} if and only if there exists $\zeta\in \E_0$ such that $\psi\in C(J,E_0)$ is the mild solution to
\begin{equation}\label{psi}
\partial_a\psi = A_\ell(a)\psi -\zeta(a)\,,\quad a\in J\,,\qquad \psi(0)=\int_0^{a_m} b_\ell(a) \psi(a)\,\rd a\,.
\end{equation}
%with 
%\begin{equation}\label{psi0}
%\psi(0)=\int_0^{a_m} b_\ell(a) \psi(a)\,\rd a\,.
%\end{equation}
In this case, $\A \psi = \zeta$.

Finally, the embedding $D(\A) \hookrightarrow \E_\alpha$ is continuous and dense for  $\alpha\in[0,1)$.
\end{thm}

\begin{proof}
This follows from \cite[Theorem~1.2, Corollary~1.3, Theorem~1.4]{WalkerIUMJ}\footnote{The positivity assumption \cite[Assumption~(1.5)]{WalkerIUMJ} is in fact not needed in order to prove \cite[Theorem~1.4]{WalkerIUMJ}, since for $\mathrm{Re}\,\lambda>0$ large,  $(1-Q_\lambda)^{-1}\in\ml(E_0)$  is ensured by the fact that   $\|Q_\lambda\|_{\ml(E_0)}< 1$.}.
\end{proof}

That $\psi\in C(J,E_0) \subset \E_0$ is the mild solution to~\eqref{psi} with $\zeta\in \E_0$ means that
$$
\psi(a)=\Pi_\ell(a,0)\psi(0)-\int_0^a\Pi_\ell(a,\sigma)\,\zeta(\sigma)\,\rd \sigma\,,\quad a\in J\,.
$$
The spectrum of the generator $\A$ has been investigated in \cite{WalkerIUMJ}  when the generated semigroup~$(\mS(t))_{t\ge 0}$  is positive.  In the following,  let
$$
s(\A):=\sup\left\{\mathrm{Re}\, \lambda\,;\, \lambda\in\sigma(\A)\right\}
$$ 
be the spectral bound of the generator $\A$  and 
$$
\omega_0(\A):=\inf\left\{\omega\in\R\,;\, \sup_{t>0}\big(e^{-\omega t}\|e^{t\A}\|_{\ml(\E_0)}\big)<\infty\right\}
$$
be the growth bound of the corresponding semigroup $\mS(t)=e^{t\A}$, $t\ge 0$.

%The next result will also be useful when considering examples later on:

\begin{prop}\label{IUMJprop}
Suppose \eqref{A1l} and \eqref{A2l}. Let $A_\ell(a)$ be resolvent positive for every $a\in J$ and  $b_\ell\in L_{\infty}\big(J,\ml_+(E_0)\big)$ such that
$b_\ell(a)\Pi_\ell(a,0)\in \ml_+(E_0)$ is strongly positive\footnote{Recall that if $E$ is an ordered Banach space, then $T\in \ml(E)$ is {\it strongly positive} if $Tz\in E$ is a {\it quasi-interior point} for each $z\in E^+\setminus\{0\}$, that is, if $\langle z',Tz\rangle_{E} >0$ for every $z'\in (E')^+\setminus\{0\}$.} for $a$ in a subset of $J$ of positive measure.
Then $$s(\A)=\omega_0(\A)=\lambda_0\,,$$ where   $\lambda_0\in\R$ is  uniquely determined from the condition $r(Q_{\lambda_0})=1$ with $r(Q_{\lambda})$ denoting for $\lambda\in\R$ the spectral radius of the strongly positive compact operator $Q_\lambda\in\ml(E_0)$ given by
\begin{equation*}
Q_\lambda :=\int_0^{a_m} e^{-\lambda a}\,b_\ell(a)\,  \Pi_\ell(a,0)\,  \rd a\,.
\end{equation*}
\end{prop}

\begin{proof}
This is \cite[Corollary~4.3]{WalkerIUMJ}.
\end{proof}
%%%%%%%%%%%%%%%%%%%%%%%%%%%%%%%%%%%%%%%%%%%%%%%%%%
%%%%%%%%%%%%%%%%%%%%%%%%%%%%%%%%%%%%%%%%%%%%%%%%%%

%%%%%%%%%%%%%%%%%%%%%%%%%%%%%%%%%%%
\subsection{Eventual Compactness of the Perturbed Semigroup}
%%%%%%%%%%%%%%%%%%%%%%%%%%%%%%%%%%%

If $\B\in\ml (\E_0)$, then \mbox{$\G:=\A+\B$} generates also a strongly continuous semigroup on $\E_0$, which, however, is not necessarily eventually compact even though the semigroup generated by $\A$ is. Nonetheless, we next shall prove that for particular (nonlocal) perturbations of the form
\begin{subequations}\label{Bpart0}
\begin{align}\label{Bpert0}
[\B\psi](a):=\int_0^{a_m} q(a,\sigma)\,\psi(\sigma)\,\rd \sigma\,,\quad a\in J\,,\quad \psi\in \E_0\,, 
\end{align}
with $q(a,\sigma)=q(a)(\sigma)$ satisfying
\begin{align}\label{q0}
q\in C\big(J,L_\infty(J,\ml(E_0))\big)\,,%\cap  L_1\big(J,L_\infty(J,\ml(E_\alpha,E_0))\big)
\end{align}
 the semigroup generated by $\G=\A+\B$ is eventually compact. This yields more information on the spectrum of $\G$ and implies, in particular, that the spectral bound
$s(\G)$
and the growth bound
$\omega_0(\G)$ coincide. This we shall apply later on to the linearization of problem~\eqref{PPPP} in which the  perturbation operator $\B$ has the particular form~\eqref{Bpart0}.
Note  that $\B\in\ml(\E_0)$ with
$$
\|\B\|_{\ml(\E_0)}\le a_m\,\|q\|_{\infty}\,,\qquad \|q\|_{\infty}:=\|q\|_{C(J,L_\infty(J,\ml(E_0)))}\,.
$$
For the birth rate we impose additionally that
\begin{equation}\label{b}
b_\ell\in C\big(J,\ml(E_0)\big)\,.
\end{equation}
\end{subequations}

We then shall prove the following theorem which is fundamental for our purpose:

%%%%%%%%%%%%%%%%%%%%%%%%%%%%%%%%%%%
%%%%%%%%%%%%%%%%%%%%%%%%%%%%%%%%%%%
\begin{thm}\label{T:NormCont}
Suppose \eqref{A1l}, \eqref{A2l},~\eqref{Bpart0}, and let  $\A$ be the generator of the semigroup $(\mS(t))_{t\ge 0}$ defined in~\eqref{100}.Then the semigroup $(\T(t))_{t\ge 0}$ on $\E_0$ generated by $\G:=\A+\B$ is eventually compact. In particular,
\begin{align}\label{SO}
s(\G)=\omega_0(\G)\,
\end{align}
and the spectrum $\sigma(\G)=\sigma_p(\G)$  is countable and consists of poles of the resolvent $R(\cdot,\G)$ of finite algebraic multiplicities (in particular, $\sigma(\G)$ is a pure point spectrum). Moreover, for each $r\in\R$, the set $\{\lambda\in \sigma(\G)\,;\,\mathrm{Re}\, \lambda\ge r\}$ is finite. 

Finally,  if $A_\ell(a)$ is resolvent positive for every $a\in J$, if $b_\ell\in L_{\infty}\big(J,\ml_+(E_0)\big)$, and if $\B\in\ml_+(\E_0)$, then the semigroup~$(\T(t))_{t\ge 0}$ is positive, and if $s(\G)>-\infty$, then $s(\G)$ is an eigenvalue of $\G$.
\end{thm}
%%%%%%%%%%%%%%%%%%%%%%%%%%%%%%%%%%%
%%%%%%%%%%%%%%%%%%%%%%%%%%%%%%%%%%%

Theorem~\ref{T:NormCont} relies on the following crucial observation:

%%%%%%%%%%%%%%%%%%%%%%%%%%%%%%%%%%%
%%%%%%%%%%%%%%%%%%%%%%%%%%%%%%%%%%%
\begin{prop}\label{Prop:NormCont}
Suppose \eqref{A1l}, \eqref{A2l}, and \eqref{Bpart0}. Set
\begin{align*} 
\mathcal{V}\mS(t)\psi:=\int_0^t\mS(t-s)\,\B\,\mS(s)\psi\,\rd s\,,\quad t>0\,,\quad \psi\in\E_0\,.
\end{align*}
Then $\mathcal{V}\mS:(0,\infty)\to \mathcal{K}(\E_0)$, i.e. $\mathcal{V}\mS(t)$ is a bounded compact operator on $\E_0$ for every $t\in (0,\infty)$.
\end{prop}
%%%%%%%%%%%%%%%%%%%%%%%%%%%%%%%%%%%
%%%%%%%%%%%%%%%%%%%%%%%%%%%%%%%%%%%

The proof of Proposition~\ref{Prop:NormCont} uses the explicit form of the semigroup $(\mS(t))_{t\ge 0}$ in~\eqref{100} and the particular form of the perturbation $\B$ in~\eqref{Bpert0}, but is rather technical and thus postponed to Appendix~\ref{App:NormCont}. 

\subsection*{Proof of Theorem~\ref{T:NormCont}}
Recall that $\mS(t)=e^{t\A}$ and $\T(t)=e^{t(\A+\B)}$.
Theorem~\ref{IUMJT1} ensures that the semigroup $(\mS(t))_{t\ge 0}$ is eventually compact. Therefore, since~$\B\in\ml(\E_0)$ and since \mbox{$\mathcal{V}\mS$} is compact on $(0,\infty)$
according to Proposition~\ref{Prop:NormCont}, we infer from~\cite[III.Theorem~1.16~(ii)]{EngelNagel} %\footnote{{\tred Here we require $\alpha=0$.}} 
 (with $k=1$) that also  $(\T(t))_{t\ge 0}$ is eventually compact. This implies~\eqref{SO} due to~\cite[IV.Corollary~3.11]{EngelNagel} while the remaining statements regarding the spectrum of~$\G=\A+\B$ now follow from~\cite[V.Corollary~3.2]{EngelNagel}.

Finally,  if $A_\ell(a)$ is resolvent positive for every $a\in J$ and $b_\ell\in L_{\infty}\big(J,\ml_+(E_0)\big)$, then Theorem~\ref{IUMJT1} entails that $\mS(t)=e^{t\A}$ is positive. Since $\B\in\ml_+(\E_0)$, it is well-known that the semigroup~$\T(t)=e^{t(\A+\B)}$ is positive as well, e.g. see \cite[Proposition~12.11]{BFR}. Since $\E_0$ is a Banach lattice, this implies that $s(\G)$ is an eigenvalue of $\G$ if $s(\G)>-\infty$ according to \cite[Corollary~12.9]{BFR}. This yields Theorem~\ref{T:NormCont}.\qed \\

We can derive now also properties of the semigroup $(\T(t))_{t\ge 0}$ restricted to $\E_\alpha$. 

%%%%%%%%%%%%%%%%%%%%%%%%%%%%%%%%%%%
%%%%%%%%%%%%%%%%%%%%%%%%%%%%%%%%%%%
\begin{cor}\label{C:Ealpha}
Suppose \eqref{A1l}, \eqref{A2l},~\eqref{Bpart0}, and let  $\A$ be the generator of the semigroup $(\mS(t))_{t\ge 0}$ defined in \eqref{100}.  Let $(\T(t))_{t\ge 0}$ be the strongly continuous semigroup on $\E_0$ generated by $\G=\A+\B$. For every $\alpha\in [0,1)$, the restriction $\left(\T(t)\vert_{\E_\alpha}\right)_{t\ge 0}$ is a strongly continuous, eventually compact semigroup on $\E_\alpha$. Its generator $\G_\alpha$ is the $\E_\alpha$-realization of $\G$. For the corresponding (point) spectra it holds that $\sigma(\G)=\sigma(\G_\alpha)$. Moreover, for every $\omega>s(\G)$ there is $N_\alpha\ge 1$ such that
\begin{align}\label{growth}
\|\T(t)\|_{\ml(\E_0)}+\|\T(t)\|_{\ml(\E_\alpha)}+t^\alpha \|\T(t)\|_{\ml(\E_0,\E_\alpha)}\le N_\alpha e^{\omega t}\,,\quad t\ge 0\,.
\end{align}
\end{cor}
%%%%%%%%%%%%%%%%%%%%%%%%%%%%%%%%%%%
%%%%%%%%%%%%%%%%%%%%%%%%%%%%%%%%%%%

\begin{proof}
{\bf (i)} It was shown in \cite[Theorem~1.2]{WalkerIUMJ}  that the semigroup $(\T(t))_{t\ge 0}$ on $\E_0$ generated by $\G=\A+\B$ is given by
\begin{equation}\label{E33T}
\T(t)\phi=\mS(t)\phi+\int_0^t\mS(t-s)\,\B\, \T(s)\,\phi\,\rd s\,,\quad t\ge 0\,,\quad \phi\in\E_0\,,
\end{equation}
and that there are $C_\alpha\ge 1$ and $\varsigma_\alpha\ge 0$ such that
\begin{equation}\label{E3T}
\|\T(t)\|_{\ml(\E_0,\E_\alpha)}\le C_\alpha\, e^{\varsigma_\alpha t} t^{-\alpha}\,,\quad t> 0\,.
\end{equation}
We then infer from~\eqref{E3} and~\eqref{E33T} that
$$
\|\T(t)\phi\|_{\E_\alpha}\le M_\alpha e^{\varkappa_\alpha t} \|\phi\|_{\E_\alpha}+M_\alpha\|\B\|_{\ml(\E_\alpha,\E_0)}\int_0^t e^{\varkappa_\alpha (t-s)}\, (t-s)^{-\alpha}\, \|\T(s)\phi\|_{\E_\alpha}\,\rd s\,,\quad t\ge 0\,, 
$$
for $\phi\in \E_\alpha$ so that Gronwall's inequality implies 
\begin{equation}\label{xd}
\|\T(t)\|_{\ml(\E_\alpha)}\le c_\alpha e^{\omega_1 t}\,,\quad t\ge 0\,,
\end{equation}
for some $c_\alpha\ge 1$ and $\omega_1>0$.
Moreover, since
$$
\|\T(t)\phi-\phi\|_{\E_\alpha}\le\|\mS(t)\phi-\phi\|_{\E_\alpha}+\int_0^t\|\mS(t-s)\|_{\ml(\E_0,\E_\alpha)}\,\|\B\|_{\ml(\E_\alpha,\E_0)}\, \|\T(s)\phi\|_{\E_\alpha}\,\rd s
$$
for $t\ge 0$ and  $\phi\in \E_\alpha$, the strong continuity of $\left(\T(t)\vert_{\E_\alpha}\right)_{t\ge 0}$ on $\E_\alpha$ follows from the  strong continuity of $(\mS(t)\vert_{\E_\alpha})_{t\ge 0}$ on $\E_\alpha$ guaranteed by Theorem~\ref{IUMJT1}  and from~\eqref{E3} and~\eqref{xd}. Therefore, $\left(\T(t)\vert_{\E_\alpha}\right)_{t\ge 0}$ is a strongly continuous semigroup on $\E_\alpha$.  Writing $$\T(t)\vert_{\E_\alpha}=\T(t-t_0)\T(t_0)\vert_{\E_\alpha}$$ and noticing that $\T(t_0)\in \ml(\E_0)$ is compact for $t_0$ large due to Theorem~\ref{T:NormCont} while \mbox{$\T(t-t_0)\in \ml(\E_0,\E_\alpha)$} by \eqref{E3T} for $t>t_0$, we deduce that $\left(\T(t)\vert_{\E_\alpha}\right)_{t\ge 0}$ is eventually compact on $\E_\alpha$.\\ 

{\bf (ii)} Denote by $\G_\alpha$  the generator of the restricted semigroup to $\E_\alpha$ so that $\T(t)\vert_{\E_\alpha}=e^{t\G_\alpha}$ for $t\ge 0$. We prove that $\G_\alpha$ is the $\E_\alpha$-realization of $\G$. To this end, denote the latter by $\G_{\E_\alpha}$ and let $\zeta\in\mathrm{dom} (\G_\alpha)$. Then
$$
\frac{1}{t}\big(\T(t)\zeta-\zeta\big)=\frac{1}{t}\left(e^{t\G_\alpha}\zeta-\zeta\right)\rightarrow \G_\alpha\zeta \ \ \text{ in }\ \E_\alpha\hookrightarrow \E_0\ \text{ as } \ t\to 0\,,
$$
hence $\zeta\in\mathrm{dom} (\G)$ and 
\begin{align}\label{EE}
\G\zeta=\G_\alpha\zeta\in \E_\alpha\,.
\end{align} 
Consequently, $\zeta\in\mathrm{dom} (\G_{\E_\alpha})$ and $\G_{\E_\alpha}\zeta=\G_\alpha\zeta$.  

Conversely, let $\psi\in\mathrm{dom} (\G_{\E_\alpha})$ and $\lambda\in \rho(\G)\cap \rho(\G_\alpha)$ (this is possible since both $\G$ and $\G_\alpha$ are semigroup generators). Then 
$$
\psi\in  \mathrm{dom} (\G)=\mathrm{dom} (\A)\subset \E_\alpha
$$ 
by Theorem~\ref{IUMJT1} and $\G\psi\in \E_\alpha$. Thus, since $(\lambda-\G)\psi\in \E_\alpha$ and since $\lambda\in \rho(\G_\alpha)$, there is a unique $\zeta\in \mathrm{dom} (\G_{\alpha})$ such that $(\lambda-\G_\alpha)\zeta=(\lambda-\G)\psi$. Since $(\lambda-\G_\alpha)\zeta=(\lambda-\G)\zeta$ by \eqref{EE} and since $\lambda\in \rho(\G)$, we conclude $\psi=\zeta\in \mathrm{dom} (\G_\alpha)$. Therefore, $\G_\alpha=\G_{\E_\alpha}$; that is,~$\G_\alpha$ coincides with the $\E_\alpha$-realization of $\G$.\\

{\bf (iii)} We claim that $\sigma(\G)=\sigma(\G_\alpha)$, where we recall that both $\sigma(\G)=\sigma_p(\G)$ and $\sigma(\G_\alpha)=\sigma_p(\G_\alpha)$ are point spectra since both $\G$ and $\G_\alpha$ generate eventually compact semigroups by Theorem~\ref{T:NormCont} respectively {\bf (i)}. Let $\lambda\in \sigma(\G)$. Then  there is $\psi\in \mathrm{dom} (\G)\setminus\{0\}$ with 
$$
\G\psi=\lambda\psi\in \mathrm{dom} (\G)=\mathrm{dom} (\A)\subset \E_\alpha
$$
by Theorem~\ref{IUMJT1}. From {\bf (ii)} we deduce that $\psi\in  \mathrm{dom} (\G_\alpha)$ with $\G_\alpha\psi=\lambda\psi$. Therefore, $\lambda\in \sigma(\G_\alpha)$. 

Conversely, let  $\mu\in \sigma(\G_\alpha)$. Then there is $\zeta\in \mathrm{dom} (\G_\alpha)\setminus\{0\}$ with 
$\G_\alpha\zeta=\mu\zeta$. Since $\G_\alpha\zeta=\G\zeta$, we conclude that $\mu\in\sigma(\G)$.\\

{\bf (iv)} Consider now $\omega>s(\G)$. Since $\G$ and $\G_\alpha$ both generate eventually compact semigroups, it follows from {\bf (iii)} and \cite[IV.~Corollary~3.11]{EngelNagel} that
$$
\omega_0(\G)=s(\G)=s(\G_\alpha)=\omega_0(\G_\alpha)\,.
$$
Thus, for $\omega-\ve>s(\G)$ with $\ve>0$ there is $c_\alpha\ge 1$ such that
\begin{align}\label{growth1}
\|\T(t)\|_{\ml(\E_0)}+\|\T(t)\|_{\ml(\E_\alpha)}\le c_\alpha e^{(\omega-\ve) t}\,,\quad t\ge 0\,.
\end{align}
On the one hand, it follows from \eqref{E3T} and \eqref{growth1} for $t> 1$ that 
\begin{align*}
\|\T(t)\|_{\ml(\E_0,\E_\alpha)}&\le  \|\T(1)\|_{\ml(\E_0,\E_\alpha)}\, \|\T(t-1)\|_{\ml(\E_0)}\,\le  C_\alpha\, e^{\varsigma_\alpha}\, c_\alpha\, e^{(\omega-\ve) (t-1)}\\
&\le 
c_\alpha \,C_\alpha\,e^{-\omega+\ve}\, \left(\sup_{s>0}s^\alpha e^{-\ve s}\right)\, t^{-\alpha}\, e^{\omega t}\,.
\end{align*}
On the other hand, due to \eqref{E3T} we have, for $0<t\le 1$,
\begin{align*}
\|\T(t)\|_{\ml(\E_0,\E_\alpha)}&\le   C_\alpha\, e^{\varsigma_\alpha} t^{-\alpha}\le 
 C_\alpha\,e^{\varsigma_\alpha}\, e^{\vert\omega\vert}\, t^{-\alpha}\, e^{\omega t}\,.
\end{align*}
Consequently, there is $N_\alpha\ge 1$ such that
$$
\|\T(t)\|_{\ml(\E_0,\E_\alpha)}\le   N_\alpha
 \, t^{-\alpha}\, e^{\omega t}\,,\quad t>0\,,
$$
and the assertion follows.
\end{proof}

\begin{rem}
One can show that  $\A+\B$ has compact resolvent for any $\B\in\ml(\E_0)$ (not necessarily satisfying~\eqref{Bpart0}). We refer to a forthcoming paper~\cite{Walker2023}.
\end{rem}

%%%%%%%%%%%%%%%%%%%%%%%%%%%%%%%%%%%
%%%%%%%%%%%%%%%%%%%%%%%%%%%%%%%%%%%
\section{Linearized Stability for the Nonlinear Problem: Proof of Theorem~\ref{MainT}}\label{Sec4}
%%%%%%%%%%%%%%%%%%%%%%%%%%%%%%%%%%%
%%%%%%%%%%%%%%%%%%%%%%%%%%%%%%%%%%%

The development of this section is based upon the treatment of linearized stability in~\cite{WalkerZehetbauerJDE}, which, in turn, is based on the treatment of the case without spatial diffusion in~\cite{Pruess83}. In fact, we follow the exquisite exposition in~\cite[Section 4.5]{WebbBook} of this case.\\ 

For the reminder of this section, let $\phi\in\E_1\cap C(J,E_\alpha)$ be a fixed equilibrium to the nonlinear problem~\eqref{PPPP} and assume \eqref{A1}, \eqref{A11}, and \eqref{B}. As pointed out in Section~\ref{Sec2} we shall derive statements on the stability or instability of $\phi$ from information on the (formally) linearized problem~\eqref{EPx}, that is, from information on
\begin{align*} 
	\partial_t v+ \partial_av \, &=     A(a)v -m\big(\bar \phi,a\big) v -\partial m\big(\bar \phi,a\big)[\bar v(t)]\phi(a)\,, \qquad t>0\, ,\quad a\in (0,a_m)\, ,\\ 
	v(t,0)&=\int_0^{a_m}b\big(\bar \phi,a\big)\, v(t,a)\, \rd a +\int_0^{a_m}\partial b\big(\bar \phi,a\big)[\bar v(t)]\, \phi(a)\, \rd a \,, \qquad t>0\, , \label{Plinear2} \\
	v(0,a)&=  v_0(a)\,, \qquad a\in (0,a_m)\,.
\end{align*}
In the following, we demonstrate that this linear problem fits into the framework of Section~\ref{Sec3}. More precisely, the solution $v$ is given by a semigroup $(\T_\phi(t))_{t\ge 0}$ generated by an (unbounded) operator of the form $\A_\phi+\B_\phi$  as in Theorem~\ref{T:NormCont} with perturbation $\B_\phi w=-\partial m(\bar\phi,a)[\bar w]\phi(a)$. Moreover, if $u(\cdot;u_0)$ is the solution to the nonlinear problem~\eqref{PPPP} provided by Proposition~\ref{T1JDE}, then the difference $u(\cdot;u_0)-\phi$ can be represented in terms of this semigroup $(\T_\phi(t))_{t\ge 0}$, see Proposition~\ref{representation} below. This will be the key for our stability and instability results stated in Theorem~\ref{MainT} (see also Theorem~\ref{TInStable} below).\\

We focus our attention on the linearization~\eqref{EPx}.
Regarding the linearized age boundary conditions~\eqref{e12b} we point out that
\begin{equation*}% 
\int_0^{a_m}  b(\bar \phi,a)\,\zeta(a)\,\rd a +\int_0^{a_m}  \partial b(\bar \phi,a)[\bar \zeta]\phi(a)\,\rd a=\int_0^{a_m} b_\phi(a)\zeta(a)\,\rd a\,,\quad \zeta\in\E_0\,,
		\end{equation*}
where we set
\begin{subequations}\label{ell}
\begin{equation}\label{bell}
b_\phi(a)v:= b(\bar \phi,a) v+\int_0^{a_m}  \partial b(\bar \phi,\sigma)[\varrho(a) v]\, \phi(\sigma)\,\rd \sigma\,,\quad a\in J\,,\quad v\in E_0\,,
	\end{equation}
so that 
$$
b_\phi\in C\big(J,\ml(E_0)\big)\cap L_{\infty}\big(J,\ml(E_\theta)\big)\,,\quad \theta\in\{0,\alpha,\vartheta\}\,,
$$ according to \eqref{B3}, \eqref{B2b}, and \eqref{A1d}.
%Also note that~\eqref{B2} entails
%\begin{equation}\label{14C}
%	\|\partial b(\bar\phi,a)[\bar v]\phi(a)\|_{E_0}\le c_b\,\|\bar v\|_{E_\alpha}\,\|\phi(a)\|_{E_0}\,,\quad a\in J\,,\quad \bar v\in E_\alpha\,,
%\end{equation}
%\end{subequations}
%with $c_b:=\|\partial b(\bar\phi,\cdot)\|_{\ml(E_\alpha,L_{\infty}(J,\ml(E_\alpha,E_0)))}$. 
We also introduce
\begin{equation}\label{Aell}
A_\phi(a)v:=A(a)v-m(\bar\phi,a)v\,,\quad v\in E_1\,,\quad a\in J\,,
	\end{equation}
and infer from \eqref{A1}, \eqref{B2bb}, and \cite[I.Theorem 1.3.1]{LQPP} that 
$$
A_\phi\in C^\rho\big(J,\mathcal{H}(E_1,E_0)\big)\,.
$$
Therefore, $A_\phi$ and $b_\phi$ satisfy \eqref{A1l}, \eqref{A2l}, and \eqref{b}.  Moreover, we define $\B_\phi\in\ml(\E_0)$ by
\begin{equation}\label{BBell}
[\B_\phi \zeta](a):=-\partial m(\bar\phi,a)[\bar \zeta]\phi(a)\,,\quad a\in J\,,\quad \zeta\in \E_0\,,
	\end{equation}
and infer from~\eqref{B1}, \eqref{B2bbb}, and \eqref{A1d} that
\begin{equation}\label{qqq}
[\B_\phi \zeta](a)=-\partial m(\bar\phi,a)\left[\int_0^{a_m}\varrho(\sigma) \zeta(\sigma)\,\rd \sigma\right]\phi(a)=\int_0^{a_m} q(a,\sigma)\zeta(\sigma)\,\rd \sigma\,,\quad a\in J\,,
	\end{equation}
for $\zeta\in \E_0$, where $q\in C\big(J,C(J,\ml(E_0))\big)$ is given by
$$
q(a,\sigma)v:=-\partial m(\bar\phi,a)[\varrho(\sigma)v]\phi(a)\,,\quad a,\sigma\in J\,,\quad v\in E_0\,.
$$
Hence, \eqref{Bpert0} and \eqref{q0} also hold, and we are in a position to  apply the results from the previous section with $A_\ell$ and $b_\ell$ replaced by $A_\phi$ respectively $b_\phi$. Throughout the reminder  of this section we thus assume (see Theorem~\ref{IUMJT1})  that
\begin{equation}\label{assump}
\begin{split}
&\textit{$\A_\phi$ is the generator of the semigroup associated with~\eqref{PP} and the data}\\
&\textit{$A_\phi$ and $b_\phi$ introduced in~\eqref{bell}-\eqref{Aell}. Moreover, $(\T_\phi(t))_{t\ge 0}$ is the strongly} \\
&\textit{continuous semigroup on~$\E_0$ generated by $\G_\phi:=\A_\phi+\B_\phi$}.
\end{split}
\end{equation} 
\end{subequations}
Theorem~\ref{T:NormCont} implies that ~$(\T_\phi(t))_{t\ge 0}$  is  eventually compact and the spectrum  consists of eigenvalues only, i.e. \mbox{$\sigma(\G_\phi)=\sigma_p(\G_\phi)$}. Moreover, the characterization of $\A_\phi$ (and thus of~$\G_\phi$) in Theorem~\ref{IUMJT1}~{\bf (c)} and \eqref{ell} imply that the eigenvalue problem for $\G_\phi$ corresponds exactly to~\eqref{eigenvalueproblem}.\vspace{1mm}

%Throughout the reminder of this section we denote by $\A_\phi$ the generator of the semigroup associated with~\eqref{PP} and the just introduced data $A_\phi$ and $b_\phi$ in~\eqref{ell} as provided by Theorem~\ref{IUMJT1}. 
%Moreover, we let $(\T_\phi(t))_{t\ge 0}$ be the strongly continuous semigroup on~$\E_0$ generated by $\G_\phi:=\A_\phi+\B_\phi$. Then Theorem~\ref{T:NormCont} implies that ~$(\T_\phi(t))_{t\ge 0}$  is  eventually compact and $\sigma(\G_\phi)=\sigma_p(\G_\phi)$ consists of eigenvalues only.\\

We shall then prove the following reformulation of Theorem~\ref{MainT} regarding the stability of equilibria to the nonlinear problem~\eqref{PPPP}:   \vspace{1mm}

\begin{thm}~\label{TInStable}
Let $\alpha\in [0,1)$ and let $\phi\in\E_1 \cap C(J,E_\alpha)$ be an equilibrium   to~\eqref{PPPP}. Assume \eqref{A1}, \eqref{A11}, \eqref{B}, and use the notation~\eqref{ell}. The following hold:
\begin{itemize}
\item[\bf (a)] If $\mathrm{Re}\,\lambda < 0$ for any $\lambda\in \sigma_p(\G_\phi)$,   then $\phi$  is exponentially asymptotically stable in~$\E_\alpha$.

\item[\bf (b)] If $\mathrm{Re}\,\lambda > 0$ for some $\lambda\in \sigma_p(\G_\phi)$,   then $\phi$  is unstable in $\E_\alpha$.
\end{itemize}
\end{thm}
 
%In order to prove  Theorem~\ref{TInStable} we impose throughout the remainder of this Section assumptions~\eqref{A1}, \eqref{A11}, \eqref{B} and use the notation~\eqref{ell}, where $\phi\in\E_1 \cap C(J,E_0)$ is a fixed equilibrium  to~\eqref{PPPP}. By $(\T_\phi(t))_{t\ge 0}$ we denote the strongly continuous semigroup on $\E_0$ generated by $\\G_\phi=A_\phi+\B_\phi$, where $\A_\phi$ is the generator of the semigroup associated with~\eqref{PP} (and data from~\eqref{ell}) introduced in~Theorem~\ref{IUMJT1}.

%%%%%%%%%%%%%%%%%%%%%%%%%%%%%%%%%%%
\subsection*{Proof of Theorem~\ref{TInStable}~(a): Stability}
%%%%%%%%%%%%%%%%%%%%%%%%%%%%%%%%%%%
%%%%%%%%%%%%%%%%%%%%%%%%%%%%%%%%%%%
Assumptions \eqref{A1}, \eqref{A11}, \eqref{B}  ensure that we are in a position to apply \cite[Theorem~2.2]{WalkerZehetbauerJDE}, where it was shown that the equilibrium $\phi$ is exponentially asymptotically stable in $\E_\alpha$ provided that there is $\omega_\alpha(\phi)>0$ such that
\begin{align}\label{c22}
\|\T_\phi(t)\|_{\ml(\E_\alpha)}+t^\alpha \|\T_\phi(t)\|_{\ml(\E_0,\E_\alpha)}\le Me^{-\omega_\alpha(\phi) t}\,,\quad t> 0\,.
\end{align}
Now, the supposition $\mathrm{Re}\,\lambda < 0$ for any $\lambda\in \sigma_p(\G_\phi)$ ensures a negative spectral bound $s(\G_\phi)<0$ so that Corollary~\ref{C:Ealpha} implies \eqref{c22}.
This proves Theorem~\ref{TInStable}~{\bf(a)}.\qed

%%%%%%%%%%%%%%%%%%%%%%%%%%%%%%%%%%%
\subsection*{Preparation of the Proof of  Theorem~\ref{TInStable}~(b): Instability}
%%%%%%%%%%%%%%%%%%%%%%%%%%%%%%%%%%%
%%%%%%%%%%%%%%%%%%%%%%%%%%%%%%%%%%%

The proof of the instability result requires some preliminaries.
First of all, we infer from the supposition of Theorem~\ref{TInStable}~{\bf (b)} and due to Theorem~\ref{T:NormCont}, that the set
$$
\Sigma_+:=\sigma(\G_\phi)\cap [\mathrm{Re}\,\lambda>0]
$$
is nonempty and finite; that is, $\Sigma_+$ is a bounded spectral set. Let 
\begin{equation}\label{defomega}
0<\omega<\inf (\mathrm{Re}\,\Sigma_+)\,.
\end{equation}
This yields the following spectral decomposition:

\begin{lem}\label{Proj}
Assume~\eqref{defomega}. There is a projection $P\in\ml(\E_0)$ yielding a decomposition
\begin{subequations}\label{proj}
\begin{equation}\label{pr1}
\E_0=\E_0^1\oplus \E_0^2\,,\qquad \{0\}\not=\E_0^1:=P(\E_0)\subset D(\G_\phi)\subset\E_\alpha\,,\qquad \E_0^2:=(1-P)(\E_0)\,,
\end{equation}
such that $\G_\phi\vert_{\E_0^1}\in \ml(\E_0^1)$. Moreover, there are $M\ge 1$ and $\delta>0$ such that
\begin{equation}\label{1499}
\|\T_\phi(t)P\|_{\ml(\E_0,\E_\alpha)}\le Me^{(\omega+\delta)t}\,,\quad t\le 0\,,
\end{equation}
and
\begin{equation}\label{14999}
\|\T_\phi(t)(1-P)\|_{\ml(\E_0,\E_\alpha)}\le M\, t^{-\alpha}\,e^{(\omega-\delta)t}\,,\quad t> 0\,.
\end{equation}
\end{subequations}
\end{lem}

\begin{proof}
Since $\Sigma_+$ is a bounded spectral set, it follows from \cite[Proposition A.1.2]{LunardiBook} (or \cite[Proposition~4.15]{WebbBook}) that there is a projection $P\in\ml(\E_0)$ such that \eqref{pr1} holds (noticing that $\mathrm{dom}(\G_\phi)=\mathrm{dom}(\A_\phi)\subset\E_\alpha$ by Theorem~\ref{IUMJT1}, see also \cite[Corollary~3.4]{WalkerIUMJ}) and such that $$\G_\phi^1:=\G_\phi\vert_{\E_0^1}\in \ml(\E_0^1)\,.$$
Moreover,  
$$
\G_\phi^2: \mathrm{dom}(\G_\phi)\cap \E_0^2\to \E_0^2\,,\quad \psi\mapsto \G_\phi\psi 
$$
with
\begin{subequations}\label{semiG}
\begin{equation}
\sigma(\G_\phi^1)=\Sigma_+\,,\qquad \sigma(\G_\phi^2)=\sigma(\G_\phi)\setminus\Sigma_+
\end{equation}
and
\begin{equation}
(\lambda-\G_\phi^i)^{-1}=(\lambda-\G_\phi)^{-1}\vert_{\E_0^i}\,,\qquad i=1,2\,,\quad \lambda\in \rho(\G_\phi)\,.
\end{equation}
\end{subequations}
Choose $\delta>0$ such that
$$
\sup\{\mathrm{Re}\,\lambda\,;\,\lambda\in \sigma(\G_\phi^2)\}<\omega-2\delta<\omega+\delta<\inf \{\mathrm{Re}\,\lambda\,;\,\lambda\in \Sigma_+\}\,.
$$
It follows from \eqref{semiG} that $\G_\phi^1$ and $\G_\phi^2$ generate strongly continuous semigroups on $\E_0^1$ respectively~$\E_0^2$ such that
\begin{equation}\label{G11}
e^{t\G_\phi^1}=e^{t\G_\phi}\vert_{\E_0^1}=e^{t\G_\phi}P\vert_{\E_0^1}\,,\qquad
e^{t\G_\phi^2}=e^{t\G_\phi}\vert_{\E_0^2}=e^{t\G_\phi}(1-P)\vert_{\E_0^2}
\end{equation}
for $t\ge 0$.
In fact, $e^{t\G_\phi^1}$ is extended to $\R$ by (see also \cite[Proposition 2.3.3]{LunardiBook})
\begin{equation}\label{G1}
e^{t\G_\phi^1}=\frac{1}{2\pi i}\int_\Gamma e^{\lambda t}\, (\lambda-\G_\phi)^{-1}\,\rd \lambda\,,\quad t\in\R\,,
\end{equation}
where $\Gamma$ is a positively oriented smooth curve in $\rho(\G_\phi)$ enclosing $\Sigma_+$ with $\mathrm{Re}\,\lambda\ge \omega+\delta$ for every $\lambda\in\Gamma$. Using~\eqref{G1}, we have, for $\psi\in \E_0$,
\begin{align*}
\|e^{t\G_\phi^1}P\psi\|_{\E_0}\le \frac{1}{2\pi }\,\vert \Gamma\vert\,\sup_{\lambda\in \Gamma}\|(\lambda-\G_\phi)^{-1}\|_{\ml(\E_0)} \,e^{(\omega+\delta) t}\,\|\psi\|_{\E_0} \,,\quad t\le 0\,.
\end{align*}
Similarly, since
$$
\G_\phi\,e^{t\G_\phi^1}=\frac{1}{2\pi i}\int_\Gamma \lambda \,e^{\lambda t}\, (\lambda-\G_\phi)^{-1}\,\rd \lambda\,,\quad t\in\R\,,
$$
we have
\begin{align*}
\|\G_\phi e^{t\G_\phi^1}P\psi\|_{\E_0}\le \frac{1}{2\pi }\,\vert \Gamma\vert\,\sup_{\lambda\in \Gamma}\|\lambda\, (\lambda-\G_\phi)^{-1}\|_{\ml(\E_0)} \,e^{(\omega+\delta) t}\,\|\psi\|_{\E_0} \,,\quad t\le 0\,.
\end{align*}
Combining the two estimates we find $N\ge 1$ such that
\begin{align*}
\|e^{t\G_\phi^1}P\psi\|_{D(\G_\phi)}=\| e^{t\G_\phi^1}P\psi\|_{\E_0}+\|\G_\phi e^{t\G_\phi^1}P\psi\|_{\E_0} \le N \,e^{(\omega+\delta) t}\,\|\psi\|_{\E_0} \,,\quad t\le 0\,.
\end{align*}
Consequently, since $D(\G_\phi)\hookrightarrow \E_\alpha$ according to \cite[Corollary~3.4]{WalkerIUMJ}, we deduce~\eqref{1499}.

Finally, since $(e^{t\G_\phi})_{t\ge 0}$ is an eventually compact semigroup on $\E_0$ by Theorem~\ref{T:NormCont}, it follows from~\eqref{G11} that also $(e^{t\G_\phi^2})_{t\ge 0}$ is an eventually compact semigroup on $\E_0^2$, hence $\omega_0(\G_\phi^2)=s(\G_\phi^2)$ due to~\cite[IV.Corollary~3.11]{EngelNagel} and therefore $\omega_0(\G_\phi^2)<\omega-2\delta$ by the choice of $\delta$. Thus, there is $N_1\ge 1$ such that
\begin{align}\label{gv1}
\|e^{t\G_\phi}(1-P)\|_{\ml(\E_0)}\le N_1\,e^{(\omega-2\delta)t}\,,\quad t\ge 0\,.
\end{align}
Noticing that also
\begin{align}\label{gv2}
\|e^{t\G_\phi}(1-P)\|_{\ml(\E_0,\E_\alpha)}\le \|e^{t\G_\phi}\|_{\ml(\E_0,\E_\alpha)}\,\|(1-P)\|_{\ml(\E_0)}\le N_2\,t^{-\alpha}\,e^{\omega_1t}\,,\quad t> 0\,,
\end{align}
 for some $\omega_1>0$ and $N_2\ge 1$ due to~\eqref{growth}, we may use \eqref{gv1}-\eqref{gv2} and argue as in part~{\bf (iv)} of the proof of Corollary~\ref{C:Ealpha} to conclude~\eqref{14999}. This proves Lemma~\ref{Proj}.
\end{proof}

For the next step we introduce for a
given function $h\in C([0,T],E_0)$ and $\gamma\in\R$ (sticking to the notation of \cite[Definition~(5.3)]{WalkerZehetbauerJDE}) the function $W_{0,0}^{\gamma,h}$  by setting
\begin{subequations}\label{49998}	
\begin{equation} 
		W_{0,0}^{\gamma,h}(t,a)\, :=\, \left\{ \begin{aligned}
			&0\, ,& &   (t,a)\in \R\times J\,,\  t\le a\, ,\\
			& e^{-\gamma a}\,\Pi_{\phi}(a,0)\, B_{0,0}^{\gamma,h}(t-a)\, ,& &  (t,a)\in [0,T]\times J\, ,\ t>a\, ,
		\end{aligned}
		\right.
	\end{equation}
where $B_{0,0}^{\gamma,h}\in C([0,T],E_0)$ satisfies
\begin{align}
		B_{0,0}^{\gamma,h}(t)\,  =\, & \int_0^t  b_\phi (a)\, e^{-\gamma a}\,\Pi_{\phi}(a,0)\, B_{0,0}^{\gamma,h}(t-a)\, \rd
			a + h(t) \,,\quad t\in [0,T]\,,
	\end{align}
\end{subequations}
with the understanding that $b_\phi(a)=0$ whenever $a\notin J$. Here, 
 $\Pi_\phi$ denotes the parabolic evolution operator associated with $A_\phi$. Let $\varpi_\phi\in \R$ be such that 
$$
 \|\Pi_\phi(a,\sigma)\|_{\ml(E_\alpha)}+ (a-\sigma)^\alpha\|\Pi_\phi(a,\sigma)\|_{\ml(E_0,E_\alpha)}\le M_*\,  e^{\varpi_\phi (a-\sigma)}\,,\quad 0\le \sigma\le a\le a_m\,,
$$
for some $M_*\ge 1$ (see \eqref{EOx}).
Then we have:

\begin{lem} \label{LemmaF}
	Let $h\in C\big([0,T],E_0\big)$, $\gamma\in\R$, and $\theta\in [0,1)$. Then $W_{0,0}^{\gamma,h}\in C((-\infty,T],\E_\theta)$ and there are constants $\mu=\mu(\phi)>0$ and $c_0=c_0(\phi)>0$ (both independent of $T$, $\gamma$, and~$h$) such that
	\begin{equation}\label{LemmaFest}
		\|W_{0,0}^{\gamma,h}(t,\cdot)\|_{\E_\theta}\le c_0\int_0^t e^{(\varpi_\phi+\mu-\gamma)(t-a)}\,(t-a)^{-\theta}\,\|h(a)\|_{E_0}\,\rd a\,,\quad t\in [0,T]\,.
	\end{equation}
\end{lem}

\begin{proof}
This is \cite[Lemma 5.7]{WalkerZehetbauerJDE}.
\end{proof}

Now, let $u_0\in\E_\alpha$ be arbitrary and set
$$
w:=u(\cdot;u_0)-\phi \,,\qquad w_0:=u_0-\phi\,,
$$
where $u(\cdot;u_0)\in C\big(I(u_0),\E_\alpha\big)$ is the unique maximal solution to the nonlinear problem~\eqref{PPPP} provided by Proposition~\ref{T1JDE}. Then, using the expansions \eqref{14Ga} and \eqref{14Ea} of $m$ respectively $b$ and the notation from~\eqref{ell}, 
 we derive that $w\in C(I(u_0),\E_\alpha)$ is the generalized solution (in the sense of~\eqref{1000}, see \cite[Proposition 4.2]{WalkerZehetbauerJDE}) to 
\begin{subequations}\label{123}
	\begin{align}
	\partial_t w+ \partial_aw \, &=     A_\phi(a)w  +[\B_\phi w(t,\cdot)](a)+R_m(w(t))\,, \qquad t\in I(u_0)\, ,\quad a\in J\, ,\\ 
	w(t,0)&=\int_0^{a_m}b_\phi(a)\, w(t,a)\, \rd a +h_w(t)\,, \qquad t\in I(u_0)\, , \\
	w(0,a)&=  w_0(a)\,, \qquad a\in J\,,
\end{align}
where $h_w\in C(I(u_0),E_0)$ is defined as
\begin{equation}\label{17Bb}
h_w(t):=\int_0^{a_m}  R_b( w(t))(a) \, \rd a\,, \quad t\in I(u_0)\,,
\end{equation}
\end{subequations} 
and reminder terms $R_m$ and $R_b$ stemming from~\eqref{B}.

The characterization of the generator $\A_\phi$ given in Theorem~\ref{IUMJT1}~{\bf (c)} gives rise to a representation of $w=u(\cdot;u_0)-\phi$ in terms of the semigroup $(\T_\phi(t))_{t\ge 0}$. In fact, the following result was established in \cite{WalkerZehetbauerJDE} (see also \cite{Pruess83} for the non-diffusive case). It is fundamental for the investigation of stability properties of the equilibrium $\phi$.

\begin{prop}\label{representation}
%Assume \eqref{A1}, \eqref{A11}, \eqref{B} and use the notation~\eqref{ell}. Let $(\T_\phi(t))_{t\ge 0}$ be the strongly continuous semigroup on $\E_0$ generated by $\A_\phi+\B_\phi$, where $\A_\phi$ is the generator of the semigroup associated with~\eqref{PP} (and data from~\eqref{ell}) introduced in~Theorem~\ref{IUMJT1}. 
Given $u_0\in \E_\alpha$  let $u(\cdot;u_0)\in C\big(I(u_0),\E_\alpha\big)$ with $I(u_0)=[0,T_{max}(u_0))$ be the unique maximal solution to the nonlinear problem~\eqref{PPPP} provided by Proposition~\ref{T1JDE}. Set \mbox{$w=u(\cdot;u_0)-\phi$} and \mbox{$w_0=u_0-\phi$}.
	Then $w\in C(I(u_0),\E_\alpha)$ can be written as
\begin{equation}
	\begin{split}\label{20}
	w(t)&=\T_\phi (t)w_0 +\int_0^t \T_\phi (t-s) \left( \big(\gamma+\B_\phi\big)W_{0,0}^{\gamma,h_w}(s,\cdot) +R_m(w(s))\right)\,\rd s + W_{0,0}^{\gamma,h_w}(t,\cdot)
\end{split}
\end{equation}
for $t\in I(u_0)$ and every $\gamma\in\R$, where $h_w\in C(I(u_0),E_0)$ stems from~\eqref{17Bb} and $$W_{0,0}^{\gamma,h_w}\in C\big((-\infty,T_{max}(u_0)),\E_\alpha\big)$$ from~\eqref{49998}.
\end{prop}

\begin{proof}
This is \cite[Proposition 6.1]{WalkerZehetbauerJDE}.
\end{proof}

In the following, assume \eqref{defomega} and let the projection $P\in\ml(\E_0)$ and the constants $M$ and $\delta$ be as in Lemma~\ref{Proj}. Further,  let $\varpi_\phi$, $c_0$, and $\mu$ be as in Lemma~\ref{LemmaF}. Choose then $\gamma>0$ such that
$$
\kappa:=\omega-\varpi_\phi-\mu+\gamma>0\,
$$
and set
\begin{equation}\label{sigma0}
R_0:=\frac{c_0\Gamma(1-\alpha)}{\kappa^{1-\alpha}}+\frac{M}{\delta}\,\big(1+\delta^\alpha\Gamma(1-\alpha)\big)\Big(\|\gamma+\B_\phi\|_{\ml(\E_0)}\frac{c_0}{\kappa}+1 \Big) \,,
\end{equation}
where $\Gamma$ denotes the Gamma function.
Recalling the function $d_o$ from \eqref{14G2} and \eqref{do2}, we may choose $r>0$ 
and  $w_0\in \E_0^1$ is such that
\begin{equation}\label{choice}
R_0\, d_o(r)<\frac{\|w_0\|_{\E_\alpha}}{r} <\frac{1}{2M}\le \frac{1}{2}\,.
\end{equation}

In order to prove the instability of $\phi$, we now show the existence of a sequence $(u_0^k)_{k\ge 1}$ such that $u_0^k\to\phi$ in $\E_\alpha$ and $$\inf_{k\ge 1}\|u(k;u_0^k)-\phi\|_{\E_\alpha}>0\,.$$ To this end, we first derive backwards solutions to problem~\eqref{PPPP}:

\begin{lem}\label{backwards}
Let $r>0$ and $w_0\in \E_0^1$ be as in \eqref{choice}.
Then, for each integer $k\ge 1$, there exists a unique function $v_k\in C\big((-\infty,k],\E_\alpha\big)$ such that
\begin{align}\label{413A0}
	v_k(t)&=\T_\phi (t-k)w_0 +W_{0,0}^{\gamma,h_{v_k}}(t,\cdot)\nonumber\\
&\quad -\int_t^k \T_\phi (t-s) P\left( \big(\gamma+\B_\phi\big)W_{0,0}^{\gamma,h_{v_k}}(s,\cdot) +R_m(v_k(s))\right)\,\rd s \\
&\quad +\int_{-\infty}^t \T_\phi (t-s) (1-P)\left( \big(\gamma+\B_\phi\big)W_{0,0}^{\gamma,h_{v_k}}(s,\cdot) +R_m(v_k(s))\right)\,\rd s \nonumber
\end{align}
for $t\le k$ and satisfying
\begin{equation}\label{413A}
\|v_k(t)\|_{\E_\alpha}\le r e^{\omega(t-k)}\,,\quad t\le k\,.
\end{equation}
\end{lem}

\begin{proof}
Let $k\ge 1$ be fixed. We introduce the complete metric space $Z=(Z,d_Z)$ by
$$
Z:=\big\{v\in C\big((-\infty,k],\E_\alpha\big)\,;\, \| v(t)\|_{\E_\alpha}\le r e^{\omega(t-k)}\,,\,t\le k\big\}
$$
equipped with the metric
$$
d_Z(v,\bar v):=\sup_{t\le k}\Big(e^{-\omega(t-k)} \|v(t)-\bar v(t)\|_{\E_\alpha} \Big)\,,\quad v, \bar v\in Z\,,
$$
and claim that
\begin{align*}
	H(v)(t)&:=\T_\phi (t-k)w_0 +W_{0,0}^{\gamma,h_{v}}(t,\cdot)\nonumber\\
&\qquad -\int_t^k \T_\phi (t-s) \,P\,\left( \big(\gamma+\B_\phi\big)W_{0,0}^{\gamma,h_{v}}(s,\cdot) +R_m(v(s))\right)\,\rd s \\
&\qquad +\int_{-\infty}^t \T_\phi (t-s) \, (1-P)\, \left( \big(\gamma+\B_\phi\big)W_{0,0}^{\gamma,h_{v}}(s,\cdot) +R_m(v(s))\right)\,\rd s \nonumber
\end{align*}
defines a contraction $H:Z\to Z$. Indeed, for $v\in Z$ we have $\|v(t)\|_{\E_\alpha}\le r$ for $t\le k$ and thus, invoking~\eqref{14G2},
\begin{subequations}\label{LemmaFest1}
\begin{equation} 
\|R_m(v(t))\|_{\E_0} \le d_o(r)\,\|v(t)\|_{\E_\alpha}\le d_o(r)\,r\, e^{\omega(t-k)}\,,\quad t\le k\,,
\end{equation}
and, together with~\eqref{17Bb},
\begin{equation*}
\|h_v(t)\|_{E_0} \le\|R_b(v(t))\|_{\E_0}\le d_o(r)\,\|v(t)\|_{\E_\alpha}\le d_o(r)\,r e^{\omega(t-k)}\,,\quad t\le k \,.
\end{equation*}
Therefore, \eqref{LemmaFest} implies for $v\in Z,$ $t\le k$, and $\theta\in\{0,\alpha\}$ that
\begin{align}
		\|W_{0,0}^{\gamma,h_v}(t,\cdot)\|_{\E_\theta}&\le \chi_{\R^+}(t)c_0d_o(r) r \int_0^t e^{(\varpi_\phi+\mu-\gamma)(t-a)}\,(t-a)^{-\theta}\,e^{\omega(a-k)}\,\rd a\nonumber\\
		&\le \frac{c_0\Gamma(1-\theta)}{\kappa^{1-\theta}}d_o(r) r e^{\omega(t-k)}\,.\label{do3}
	\end{align}
\end{subequations}
We then use~\eqref{proj},~\eqref{LemmaFest1}, and \eqref{sigma0} to derive
\begin{align*}
&	\|H(v)(t)\|_{\E_\alpha} \\
& \le \|\T_\phi (t-k)P\|_{\ml(\E_\alpha)}\,\| w_0\|_{\E_\alpha} +\|W_{0,0}^{\gamma,h_{v}}(t,\cdot)\|_{\E_\alpha}\\
&\quad +\int_t^k \|\T_\phi (t-s)P\|_{\ml(\E_0,\E_\alpha)} \Big( \|\gamma+\B_\phi\|_{\ml(\E_0)}\,\|W_{0,0}^{\gamma,h_{v}}(s,\cdot)\|_{\E_0} +\|R_m(v(s))\|_{\E_0}\Big)\,\rd s \\
&\quad +\int_{-\infty}^t \|\T_\phi (t-s)(1-P)\|_{\ml(\E_0,\E_\alpha)}  \Big( \|\gamma+\B_\phi\|_{\ml(\E_0)}\,\|W_{0,0}^{\gamma,h_{v}}(s,\cdot)\|_{\E_0} +\|R_m(v(s))\|_{\E_0}\Big)\,\rd s \nonumber\\
%\end{align*}
%\begin{align*}
&\le M e^{(\omega+\delta)(t-k)}\,\| w_0\|_{\E_\alpha} +\frac{c_0\Gamma(1-\alpha)}{\kappa^{1-\alpha}}\,d_o(r)\, r\, e^{\omega(t-k)}\\
&\quad +M\, r\,d_o(r) \Big( \|\gamma+\B_\phi\|_{\ml(\E_0)}\,\frac{c_0}{\kappa}  +1\Big) \int_t^k e^{(\omega+\delta)(t-s)}\,e^{\omega(s-k)}\,\rd s \\
&\quad +M\, r\,d_o(r) \Big( \|\gamma+\B_\phi\|_{\ml(\E_0)}\,\frac{c_0}{\kappa} +1\Big) \int_{-\infty}^t e^{(\omega-\delta)(t-s)}\, e^{\omega(s-k)}\,(t-s)^{-\alpha}\,\rd s%\\
\end{align*}
\begin{align*}
&\hspace{-3.5cm}\le M e^{\omega(t-k)}\,\| w_0\|_{\E_\alpha} +\frac{c_0\Gamma(1-\alpha)}{\kappa^{1-\alpha}}\,d_o(r)\, r\, e^{\omega(t-k)}\\
&\hspace{-3.5cm}\quad + \frac{M}{\delta}\, r\,d_o(r)\,\big(1+\delta^\alpha \Gamma(1-\alpha)\big)  \Big( \|\gamma+\B_\phi\|_{\ml(\E_0)}\,\frac{c_0}{\kappa}  +1\Big) e^{\omega(t-k)}\\
&\hspace{-3.5cm}= \left(\frac{M}{r}\| w_0\|_{\E_0^1} + R_0d_o(r)\right)\, r\,e^{\omega(t-k)}
\end{align*}
so that \eqref{choice} implies
\begin{align*}
	\|H(v)(t)\|_{\E_\alpha}\le\, r\,e^{\omega(t-k)}\,,\qquad v\in Z\,,\quad t\le k\,.
\end{align*}
That is, $H:Z\to Z$. Next, notice from~\eqref{do2} that, for $v_1, v_2\in Z$,
\begin{subequations}\label{LemmaFest2}
\begin{equation} 
\|R_m(v_1(t))-R_m(v_2(t))\|_{\E_0} \le  d_o(r)\,d_Z(v_1,v_2)\, e^{\omega(t-k)}\,,\quad t\le k\,,
\end{equation}
while from~\eqref{17Bb} and ~\eqref{do2},
\begin{equation*}
\|h_{v_1}(t)-h_{v_2}(t)\|_{E_0} \le\|R_b(v_1(t))-R_b(v_2(t))\|_{\E_0}\le d_o(r)\,d_Z(v_1,v_2)\, e^{\omega(t-k)}\,,\quad t\le k \,.
\end{equation*}
Therefore, since the mapping $[h\to W_{0,0}^{\gamma,h}]$ is linear, it follows from~\eqref{LemmaFest} that
\begin{align}
		\|W_{0,0}^{\gamma,h_{v_1}}(t,\cdot)-W_{0,0}^{\gamma,h_{v_2}}(t,\cdot)\|_{\E_\alpha}
		\le \frac{c_0\Gamma(1-\alpha)}{\kappa^{1-\alpha}}d_o(r) d_Z(v_1,v_2)  e^{\omega(t-k)}\,,\quad t\le k\,.
	\end{align}
\end{subequations}
Using then \eqref{LemmaFest2} we derive similarly as above that, for $v_1, v_2\in Z$,
\begin{align*}
	\|H(v_1)(t)-H(v_2)(t)\|_{\E_\alpha} & \le R_0\, d_o(r)\, d_Z(v_1,v_2)\,\,e^{\omega(t-k)}\,,\quad t\le k\,,
\end{align*}
so that \eqref{choice} implies
\begin{align*}
	d_Z\big(H(v_1),H(v_2)\big)\le\,  \frac{1}{2}\, d_Z(v_1,v_2)\,,\quad v_1, v_2\in Z \,.
\end{align*}
Consequently, $H:Z\to Z$ is indeed a contraction, and Lemma~\ref{backwards} follows from Banach's fixed point theorem.
\end{proof}

In fact, for positive times we have a simpler representation of $v_k$:

\begin{cor}\label{backwards2}
Let $r>0$ and $w_0\in \E_0^1$ be as in \eqref{choice}. Then $v_k\in C\big((-\infty,k],\E_\alpha\big)$ from Lemma~\ref{backwards} satisfies
\begin{align*} 
	v_k(t)&=W_{0,0}^{\gamma,h_{v_k}}(t,\cdot)+\T_\phi (t)v_k(0)\\
&\quad  +\int_0^t \T_\phi (t-s) \left( \big(\gamma+\B_\phi\big)W_{0,0}^{\gamma,h_{v_k}}(s,\cdot) +R_m(v_k(s))\right)\,\rd s 
\end{align*}
for $0\le t\le k$.
\end{cor}

\begin{proof}
Define for $t\le k$
$$
q(t):=\big(\gamma+\B_\phi\big)W_{0,0}^{\gamma,h_{v_k}}(t,\cdot) +R_m(v_k(t))\,,\qquad p(t):=W_{0,0}^{\gamma,h_{v_k}}(t,\cdot)
$$
so that
\begin{align*} 
v_k(t)-p(t)&=\T_\phi (t-k)w_0  
-\int_t^k \T_\phi (t-s) \,P\,q(s)\,\rd s  +\int_{-\infty}^t \T_\phi (t-s) \, (1-P)\, q(s)\,\rd s\,.  
\end{align*}
 Then $q\in C((-\infty,k],\E_0)$  and $p\in C((-\infty,k],\E_\alpha)$ by Lemma~\ref{LemmaF}. However, we may  approximate $q,p$ uniformly on compact intervals by continuously differentiable functions with compact support and $w_0$ by a sequence in $\dom(\G_\phi)$ to justify the formal computation
\begin{align*} 
	\frac{\rd}{\rd t}\big(v_k(t)-p(t)\big)&=\G_\phi \,\T_\phi (t-k)w_0  +Pq(t)\\
&\quad -\int_t^k \G_\phi \T_\phi (t-s)  P \,q(s)\,\rd s \\
&\quad +\int_{-\infty}^t \G_\phi \T_\phi (t-s)  (1-P) q(s)\,\rd s +(1-P) q(t) \\
&=\G_\phi\big(v_k(t)-p(t)\big) +q(t)\,.
\end{align*}
Thus, for $0\le t\le k$, 
\begin{align*} 
v_k(t)-p(t)&= \T_\phi (t)\big(v_k(0)-p(0)\big)+\int_0^t \T_\phi (t-s) q(s)\,\rd s
\end{align*}
and since $p(0)=W_{0,0}^{\gamma,h_{v_k}}(0,\cdot)=0$, the assertion follows.
\end{proof}

 We are now in a position to provide the proof of Theorem~\ref{TInStable}~{\bf (b)}.

\subsection*{Proof of Instability: Theorem~\ref{TInStable}~{\bf (b)}} In order to prove instability, we may assume without loss of generality that all solutions $u(\cdot;u_0)$ to~\eqref{PPPP} provided by Proposition~\ref{T1JDE} exist globally -- that is, $T_{max}(u_0)=\infty$ -- whenever the initial values $u_0$ are close to the equilibrium $\phi$.  We set
$$
u_0^k:=v_k(0)+\phi\,, \quad k\ge 1\,,
$$
and note from~\eqref{413A} that
\begin{equation*}\label{u0}
\|u_0^k-\phi\|_{\E_\alpha}=\|v_k(0)\|_{\E_\alpha}\le r e^{-\omega k}\rightarrow 0\,,\quad m\to\infty\,.
\end{equation*}
Hence $T_{max}(u_0^k)=\infty$ as just agreed. Let  $w_k:=u(\cdot;u_0^k)-\phi\in C(\R^+,\E_\alpha)$. Then Lemma~\ref{backwards2} and Proposition~\ref{representation} entail that both $w_k, v_k\in C([0,k],\E_\alpha)$
satisfy the fixed point equation
\begin{equation*}
	\begin{split}
	z(t)&=\T_\phi (t)\big(u_0^k-\phi\big) +\int_0^t \T_\phi (t-s) \left( \big(\gamma+\B_\phi\big)W_{0,0}^{\gamma,h_z}(s,\cdot) +R_m(z(s))\right)\,\rd s + W_{0,0}^{\gamma,h_z}(t,\cdot)
\end{split}
\end{equation*}
for $t\in [0,k]$. It is easily seen that Gronwall's inequality ensures uniqueness in $C([0,k],\E_\alpha)$ of this fixed point equation, hence $v_k=w_k$ on $[0,k]$. Thus, we deduce from~\eqref{413A0},~\eqref{proj}, and~\eqref{LemmaFest1} that
\begin{align*} 
\|u(k;u_0^k)&-\phi\|_{\E_\alpha}= \|w_k(k)\|_{\E_\alpha}=\|v_k(k)\|_{\E_\alpha}\\
&\ge \|w_0\|_{\E_\alpha}- \|W_{0,0}^{\gamma,h_{v_k}}(k,\cdot)\|_{\E_\alpha} \\
&\quad -
\int_{-\infty}^k  \|\T_\phi (k-s)(1-P)\|_{\ml(\E_0,\E_\alpha)} \\
&\qquad\qquad \times \left( \|\gamma+\B_\phi\|_{\ml(\E_0)}\,\|W_{0,0}^{\gamma,h_{v_k}}(s,\cdot)\|_{\E_0} +\|R_m(v_k(s))\|_{\E_0}\right)\,\rd s\\
&\ge \|w_0\|_{\E_\alpha}- \frac{c_0\Gamma(1-\alpha)}{\kappa^{1-\alpha}}\,d_o(r) r\\
&\quad - M  \left( \|\gamma+\B_\phi\|_{\ml(\E_0)}\,\frac{c_0}{\kappa} +1 \right) d_o(r) r
\int_{-\infty}^k e^{(\omega-\delta)(k-s)}  e^{\omega(s-k)}\, (k-s)^{-\alpha}\rd s\\
&= \|w_0\|_{\E_\alpha}- \frac{c_0\Gamma(1-\alpha)}{\kappa^{1-\alpha}}d_o(r) r \\
&\quad- \frac{M}{\delta^{1-\alpha}} \Gamma(1-\alpha) \left( \|\gamma+\B_\phi\|_{\ml(\E_0)}\frac{c_0}{\kappa} +1 \right)d_o(r) r=:\xi_0\,,
\end{align*}
where, due to~\eqref{choice} and~\eqref{sigma0},
$$
\xi_0> \left(R_0 -\frac{c_0\Gamma(1-\alpha)}{\kappa^{1-\alpha}} - \frac{M}{\delta^{1-\alpha}} \Gamma(1-\alpha) \left( \|\gamma+\B_\phi\|_{\ml(\E_0)}\,\frac{c_0}{\kappa} +1 \right)\right)\, d_o(r)\, r>0\,.
$$
Consequently, we have shown that there exists a sequence $(u_0^k)_{k\ge 1}$ such that \mbox{$u_0^k\to\phi$} in~$\E_\alpha$ as $k\to \infty$ while \mbox{$\|u(k;u_0^k)-\phi\|_{\E_\alpha}\ge \xi_0$} for $k\ge 1$. 
This proves that $\phi$ is unstable in~$\E_\alpha$ and thus Theorem~\ref{TInStable}~{\bf (b)}.\qed

%%%%%%%%%%%%%%%%%%%%%%%%%%%%%%%%%%%
%%%%%%%%%%%%%%%%%%%%%%%%%%%%%%%%%%%
\subsection*{Rephrasing the Eigenvalue Problem}\label{SecRephrase}
%%%%%%%%%%%%%%%%%%%%%%%%%%%%%%%%%%%
%%%%%%%%%%%%%%%%%%%%%%%%%%%%%%%%%%%

According to Theorem~\ref{TInStable}, the stability of an equilibrium $\phi$ is determined from the eigenvalues of the operator $\G_\phi=\A_\phi+\B_\phi$.  Clearly, $\lambda\in\C$ is an eigenvalue of $\G_\phi=\A_\phi+\B_\phi$ if and only if there is some $\psi\in\mathrm{dom}(\A_\phi)$ such that $(\lambda-\A_\phi-\B_\phi)\psi=0$.
Now, due to Theorem~\ref{IUMJT1}~{\bf (c)} and \eqref{ell}, this is equivalent to $\psi\in C(J,E_0)$ solving (in a mild sense)
\begin{subequations}\label{hu}
\begin{align}
\partial_a\psi&=\big(-\lambda+A_\phi (a)\big)\psi +(\B_\phi \psi)(a)\,,\quad a\in J\,,\label{hu1}\\
\psi(0)&=\int_0^{a_m}b_\phi(a)\,\psi(a)\,\rd a\,.\label{hu2}
\end{align}
\end{subequations}
Note that \eqref{hu1} entails
\begin{align}\label{psi0EV2}
\psi(a)=e^{-\lambda a}\Pi_{\phi}(a,0)\psi(0)+\int_0^a e^{-\lambda (a-\sigma)}\,\Pi_{\phi}(a,\sigma)(\B_\phi \psi)(\sigma)\,\rd\sigma\,,\quad a\in J\,,
\end{align}
which, when plugged into \eqref{hu2}, yields
\begin{align}\label{psi0EV1}
\left(1-\int_0^{a_m}b_\phi(a)\, e^{-\lambda a}\,\Pi_{\phi}(a,0)\,\rd a\right)\psi(0)&= \int_0^{a_m}b_\phi(a)\int_0^ae^{-\lambda (a-\sigma)}\,\Pi_{\phi}(a,\sigma)(\B_\phi \psi)(\sigma)\,\rd\sigma\rd a\,.
\end{align}
Recall from  \eqref{BBell} that
$$
[\B_\phi \psi](a)=-\partial m(\bar\phi,a)[\bar \psi]\phi(a)\,,\quad a\in J\,.
$$
We thus introduce
\begin{align}\label{KL}
K_{\phi,\lambda}(a)\bar\psi:=\int_0^ae^{-\lambda (a-\sigma)}\,\Pi_{\phi}(a,\sigma)\,\partial m(\bar\phi,\sigma)[\bar \psi]\phi(\sigma)\,\rd \sigma\,,\quad a\in J\,,
\end{align}
and then obtain from \eqref{psi0EV1} that
\begin{align*} 
\left(1-\int_0^{a_m}b_\phi(a)\, e^{-\lambda a}\,\Pi_{\phi}(a,0)\,\rd a\right)\psi(0)&= -\int_0^{a_m}b_\phi(a)K_{\phi,\lambda}(a)\,\rd a\,\bar\psi\,.
\end{align*}
Moreover, \eqref{psi0EV2} implies that $\bar\psi$ satisfies
\begin{align*} 
\left(1+\int_0^{a_m}\varrho(a) K_{\phi,\lambda}(a)\,\rd a\right)\bar\psi=\int_0^{a_m}e^{-\lambda a}\,\varrho(a) \Pi_{\phi}(a,0)\,\rd a\, \psi(0)\,.
\end{align*}
Therefore, $\lambda\in\C$ is an eigenvalue of $\G_\phi=\A_\phi+\B_\phi$ if and only if there is a nontrivial eigenvector $(\psi(0),\bar\psi)\in E_0\times E_0$ in the sense  that
\begin{equation}\label{eigenvals}
\begin{pmatrix}
1-\displaystyle\int_0^{a_m}e^{-\lambda a}\,b_\phi(a)\, \Pi_{\phi}(a,0)\,\rd a & \displaystyle\int_0^{a_m}b_\phi(a)K_{\phi,\lambda}(a)\,\rd a \\[10pt]
-\displaystyle\int_0^{a_m}e^{-\lambda a}\,\varrho(a) \Pi_{\phi}(a,0)\,\rd a &  1+\displaystyle\int_0^{a_m}\varrho(a) K_{\phi,\lambda}(a)\,\rd a
\end{pmatrix} \begin{pmatrix}
\psi(0) \\[3pt] \bar\psi
\end{pmatrix}=0\,,
\end{equation}
where  $K_{\phi,\lambda}$ is defined in \eqref{KL} with $\Pi_{\phi}$ denoting the evolution operator associated with $A_\phi$ given by
\begin{equation*}
A_\phi(a)v=A(a)v-m(\bar\phi,a)v\,,\quad v\in E_1\,,\quad a\in J\,,
\end{equation*}
and
$$
b_\phi(a)v= b(\bar \phi,a) v+\int_0^{a_m}  \partial b(\bar \phi,\sigma)[\varrho(a) v]\, \phi(\sigma)\,\rd \sigma\,,\quad a\in J\,,\quad v\in E_0\,.
$$
Consequently, we obtain from Theorem~\ref{TInStable}:\vspace{2mm}

%%%%%%%%%%%%%%%%%%%%%%%%%%
%%%%%%%%%%%%%%%%%%%%%%%%%%
\begin{prop}~\label{P47}
Let $\alpha\in [0,1)$. Assume \eqref{A1}, \eqref{A11}, \eqref{B} and let $\phi\in\E_1 \cap C(J,E_\alpha)$ be an equilibrium  to~\eqref{PPPP}. The following hold:
\begin{itemize}
\item[\bf (a)] If $\mathrm{Re}\,\lambda < 0$ for every $\lambda\in\C$ for which there is a nontrivial \mbox{$(\psi(0),\bar\psi)\in E_0\times E_0$} satisfying~\eqref{eigenvals}, then the equilibrium $\phi$  is exponentially asymptotically stable in~$\E_\alpha$.

\item[\bf (b)] If there are  $\lambda\in\C$ with $\mathrm{Re}\,\lambda > 0$ and a nontrivial $(\psi(0),\bar\psi)\in E_0\times E_0$ satisfying~\eqref{eigenvals}, then the equilibrium $\phi$  is unstable in $\E_\alpha$.
\end{itemize}
\end{prop} 
%%%%%%%%%%%%%%%%%%%%%%%%%%
%%%%%%%%%%%%%%%%%%%%%%%%%%

It is worth emphasizing that the (spectral radius of the) compact operator
\begin{align}\label{QQ}
Q_{\lambda}(\phi):=\int_0^{a_m}e^{-\lambda a}\,b_\phi(a)\, \Pi_{\phi}(a,0)\,\rd a\in\mathcal{K}(E_0)
\end{align}
occurring in the linear eigenvalue problem~\eqref{eigenvals} plays a particular role in the analysis. This becomes also apparent in the next section where we will focus on applications. For $\lambda=0$ one may interpret its spectral radius $r(Q_{0}(\phi))$ as the expected number of offspring per individual during its life span at equilibrium.

%%%%%%%%%%%%%%%%%%%%%%%%%%%%%%%%%%%%%%%%%%%%%%%%%%%%
%%%%%%%%%%%%%%%%%%%%%%%%%%%%%%%%%%%%%%%%%%%%%%%%%%%%
\section{Examples}\label{Sec5}
%%%%%%%%%%%%%%%%%%%%%%%%%%%%%%%%%%%%%%%%%%%%%%%%%%%%
%%%%%%%%%%%%%%%%%%%%%%%%%%%%%%%%%%%%%%%%%%%%%%%%%%%%

In order to shed some light on the previous results,  we consider  concrete examples. We impose  for simplicity stronger assumptions than actually required\footnote{In particular, the results require less regularity assumptions than imposed in \eqref{j} and also apply to any other uniformly elliptic second-order differential operator.}.
In the following, $\Omega\subset\R^n$ is a bounded domain with smooth boundary and outer unit normal~$\nu$. We consider
\begin{subequations}	\label{P99}
	\begin{align}
		\partial_t u+\partial_a u&=\mathrm{div}_x\big(d(a,x)\nabla_xu\big)-m\big(\bar u(t,x),a,x\big)u\ , \quad t>0\,, \quad  (a,x)\in (0,a_m)\times\Om\, ,\\
		u(t,0,x)&=\int_0^{a_m} b\big(\bar u(t,x),a,x\big)u(t,a,x)\,\rd a\, ,\quad t>0\, ,\quad x\in\Om\, ,\\
		\mathcal{B} u(t,a,x)&=0\ ,\quad t>0\, , \quad  (a,x)\in (0,a_m)\times\partial\Om\, ,\\
		u(0,a,x)&=u_0(a,x)\ ,\quad (a,x)\in (0,a_m)\times\Om\,,
	\end{align}
\end{subequations}
where 
$$
\mathcal{B}u:=(1-\delta)u+\delta \partial_\nu u\,,\quad \delta\in\{0,1\}\,,
$$ 
either refers to Dirichlet boundary conditions $u\vert_{\partial\Omega}=0$ if $\delta=0$ or Neumann boundary conditions  $\partial_\nu u=0$ if $\delta=1$ and 
$$
\bar u (t,x)=\int_0^{a_m} \varrho(a,x)\,u(t,a,x)\,\rd a\,,\qquad t\ge 0\,,\quad x\in\Omega\,.
$$
We set $J=[0,a_m]$ and assume for the data  that
\begin{subequations}\label{j}
\begin{align}
&q>n\,,\quad 2\alpha\in (n/q,2)\setminus\{\delta+1/q\}\,,\quad \rho>0\,,\label{j00}\\
&d\in C^{\rho,1}\big(J\times\bar\Omega,(0,\infty)\big)\,,\label{j1}\\
&m\in C^{4,\rho,2}\big(\R\times J\times\bar\Omega ,\R^+\big)\,,\label{j2a}\\
&b\in C^{4,0,2}\big(\R\times J\times\bar\Omega ,(0,\infty)\big)\,, \label{j2b}\\
&\varrho\in C^{0,2}\big(J\times\bar\Omega,\R^+\big)\,.\label{j3}
\end{align}
\end{subequations}
Note that one may choose $2\alpha=1$ in the following. We introduce $E_0:=L_q(\Omega)$ and 
$$
E_1:=W_{q,\mathcal{B}}^2(\Omega):=\{v\in W_{q}^2(\Omega)\,;\; \mathcal{B} w=0 \text{ on } \partial\Omega\}\,.
$$ 
Then $E_1$ is compactly embedded in the Banach lattice $E_0$ and, for real interpolation with $\theta\in (0,1)\setminus\{1/2\}$,
\begin{align*}%\label{interpol}
E_\theta&:=\big(L_q(\Omega),W_{q,\mathcal{B}}^2(\Omega)\big)_{\theta,q} \doteq W_{q,\mathcal{B}}^{2\theta}(\Omega)\\
&:=\left\{\begin{array}{ll} \left\{v\in W_{q}^{2\theta}(\Omega)\,;\; \mathcal{B} w=0 \text{ on } \partial\Omega\right\}\,, & \delta+1/q<2\theta\le 2\,,\\[3pt]
	 W_{q}^{2\theta}(\Omega)\,, & 0\le 2\theta<\delta+1/q\,,\end{array} \right.
\end{align*}
while, for complex interpolation with $\theta=1/2$, $$E_{1/2}:=\big[L_q(\Omega),W_{q,\mathcal{B}}^2(\Omega)\big]_{1/2}\doteq W_{q,\mathcal{B}}^1(\Omega)
\,.
$$
Setting
$$
A(a,x)w:=\mathrm{div}_x\big(d(a,x)\nabla_xw\big)\, ,\quad w\in W_{q,\mathcal{B}}^2(\Omega)\,,\quad a\in J\,,\quad x\in\Omega\,,
$$
it follows from~\eqref{j1} and e.g. \cite{AmannIsrael} that $A \in  C^\rho\big(J,\mathcal{H}\big(W_{q,\mathcal{B}}^2(\Omega),L_q(\Omega)\big)\big)$ while the maximum principle ensures that $A(a)$ is resolvent positive for each $a\in J$. Therefore,~\eqref{A1} holds. It follows from~\eqref{j00},~\eqref{j2a}, \eqref{j2b}, and \cite[Proposition 4.1]{WalkerAMPA} that
$$
[v\mapsto b(v,\cdot ,\cdot)]\,,\, [v\mapsto m(v,\cdot,\cdot )] \in C^1\big(W_{q,\mathcal{B}}^{2\alpha}(\Omega),L_\infty(J,W_{q,\mathcal{B}}^{2\eta}(\Omega))\big)\,,\quad 2\eta\in (0,2\alpha)\setminus\{ \delta+1/q\}\,,
$$
with
\begin{align*}%\label{diff}
\big(\partial b(v,\cdot )[h]\big)(a)(x)=\partial_1 b(v(x),a,x)h(x)\,,\qquad (a,x)\in J\times \Om\,, \quad  v, h\in W_{q,\mathcal{B}}^{2\alpha}(\Omega)\,,
\end{align*}
and correspondingly for $m$.
In particular, using the continuity of pointwise multiplication  
$$
W_{q,\mathcal{B}}^{2\eta}(\Omega)\times W_{q,\mathcal{B}}^{2\alpha}(\Omega)\to L_q(\Omega)
\,,$$ 
we deduce that~\eqref{B1}-\eqref{do2} are valid and hence also~\eqref{A1b} and~\eqref{A1c}. Clearly, \eqref{j3} implies~\eqref{A1d}. Moreover, if $\phi\in \E_1=L_1\big(J,W_{q,\mathcal{B}}^2(\Omega)\big)$ is an equilibrium to~\eqref{P99}, then 
$$
\bar\phi=\int_0^{a_m} \varrho(a,\cdot)\,\phi(a,\cdot)\,\rd a\in W_{q,\mathcal{B}}^2(\Omega)
$$
due to~\eqref{j3}, hence $b(\bar\phi,\cdot,\cdot) \in L_\infty\big(J, W_{q,\mathcal{B}}^2(\Omega)\big)$. Since pointwise multiplication 
$$
W_{q,\mathcal{B}}^{2}(\Omega)\times W_{q,\mathcal{B}}^{2\alpha}(\Omega)\to W_{q,\mathcal{B}}^{2\alpha}(\Omega)
$$ 
is continuous \cite{AmannMultiplication}, we deduce~\eqref{B3}. Moreover, since $\partial_1 b(\bar\phi,\cdot,\cdot)\in L_\infty\big(J,W_{q,N}^{2-\ve}(\Omega)\big)$ for every \mbox{$\ve>0$} small and since pointwise multiplication $$W_{q,N}^{2-\ve}(\Omega)\times W_{q,N}^{2\theta}(\Omega)\to W_{q,N}^{2\theta}(\Omega)$$ is continuous for $\theta=0,\alpha$, we obtain~\eqref{B2b} and similarly~\eqref{B2bbb}. Clearly, \eqref{j2a} implies~\eqref{B2bb}. 
Consequently, assumptions~\eqref{A1},~\eqref{A11}, and~\eqref{B} are all satisfied owing to~\eqref{j}.

Recall for $(a,x)\in J\times\Omega$ that
$$
A_\phi(a,x)w=\mathrm{div}_x\big(d(a,x)\nabla_xw\big)-m(\bar\phi(x),a,x)w\, ,\quad w\in W_{q,\mathcal{B}}^2(\Omega)\,,
$$
and
$$
b_\phi(a,x)= b(\bar \phi(x),a,x) +\int_0^{a_m}  \partial_1 b(\bar \phi(x),\sigma,x)\, \phi(\sigma,x)\,\rd \sigma\, \varrho(a,x) \,.
$$
Moreover,
$$
(\B_\phi \psi)(a,x)=-\partial_1 m(\bar\phi(x),a,x)\,\phi(a,x)\int_0^{a_m}\varrho(\sigma,x)\,\psi(\sigma,x)\,\rd \sigma\,,\quad (a,x)\in J\times\Omega\,.
$$
%and
%\begin{align*}
%(K_{\phi,\lambda}(a)\bar\psi):=\int_0^ae^{-\lambda (a-\sigma)}\,\Pi_{\phi}(a,\sigma)\,\partial_1 %m(\bar\phi(x),\sigma)\phi(\sigma,x)\,\rd \sigma\,\int_0^s\varrho(s)\,\psi(s)\,\rd s\,,\quad a\in J\,.
%\end{align*}

%%%%%%%%%%%%%%%%%%%%%%%%%%%%%%%%%%%%%%%%%%%%%%%%%%%%
%%%%%%%%%%%%%%%%%%%%%%%%%%%%%%%%%%%%%%%%%%%%%%%%%%%%
\subsection*{The Trivial Equilibrium}
%%%%%%%%%%%%%%%%%%%%%%%%%%%%%%%%%%%%%%%%%%%%%%%%%%%%
%%%%%%%%%%%%%%%%%%%%%%%%%%%%%%%%%%%%%%%%%%%%%%%%%%%%

For the particular case of the trivial equilibrium $\phi=0$, we observe that (with dot referring to the suppressed $x$-variable)
$$
b_0(a)=b(0,a,\cdot)\,,\qquad A_0(a)=A(a)-m(0,a,\cdot)\,.
$$ 
Then $v(a,\cdot)=\Pi_0(a,0)v_0$, $a\in J$,  is for each $v_0\in L_q(\Omega)$ the unique strong solution to the heat equation
$$
\partial_a v=\mathrm{div}_x\big(d(a,x)\nabla_xv\big)-m(0,a,x)v\,,\quad (a,x)\in J\times \Omega\,,\qquad v(0,x)=v_0(x)\,,\quad x\in\Omega\,,
$$
subject to Dirichlet boundary conditions if $\delta=0$ or Neumann boundary conditions if $\delta=1$.
Also note that $\B_0=0$ and $K_{0,\lambda}=0$ in \eqref{KL}.
The eigenvalue equation \eqref{eigenvals} then reduces to 
\begin{align*}%\label{psi0EV11}
(1-Q_{\lambda}(0))\psi(0)&= 0\,,
\end{align*}
where
\begin{align*}
Q_{\lambda}(0)=
\int_0^{a_m}e^{-\lambda a}\,b(0,a,\cdot)\, \Pi_{0}(a,0)\,\rd a
\end{align*}
is a strongly positive compact operator on $L_q(\Omega)$ for $\lambda\in\R$ due to \cite[Corollary~13.6]{DanersKochMedina} and the strict positivity of $b(0,\cdot)$  assumed in~\eqref{j2b}. As for its spectral radius $r(Q_{\lambda}(0))$ we note that the mapping $$\R\to (0,\infty)\,,\quad \lambda\mapsto r(Q_{\lambda}(0))$$  is continuous and strictly decreasing with 
$$
\lim_{\lambda\to-\infty}r(Q_{\lambda}(0))=\infty\,,\qquad \lim_{\lambda\to\infty}r(Q_{\lambda}(0))=0
$$
 according to \cite[Lemma~3.1]{WalkerIUMJ}. Thus, there is a unique $\lambda_0\in\R$ such that $r(Q_{\lambda_0}(0))=1$. In fact, it follows from \cite[Proposition~4.2]{WalkerIUMJ} that $\lambda_0=s(\G_0)$; that is, $\lambda_0$ coincides with the spectral bound of the generator $\G_0=\A_0$ (see Proposition~\ref{IUMJprop}).

Consequently, we can state the  stability of the trivial equilibrium according to Proposition~\ref{P47} as follows:
	
%%%%%%%%%%%%%%%%%%%%%%%%%%
%%%%%%%%%%%%%%%%%%%%%%%%%%
\begin{prop}\label{P21}
Assume \eqref{j}. Then:
\begin{itemize}
\item[\bf (a)] If $r(Q_{0}(0))<1$, then the trivial equilibrium $\phi=0$ to~\eqref{P99} is exponentially asymptotically stable in the phase space~$L_1\big(J,W_{q,\mathcal{B}}^{2\alpha}(\Omega)\big)$.

\item[\bf (b)] If $r(Q_{0}(0))>1$, then the trivial equilibrium $\phi=0$ to~\eqref{P99} is unstable in the phase space~$L_1\big(J,W_{q,\mathcal{B}}^{2\alpha}(\Omega)\big)$.
\end{itemize}
\end{prop} 
%%%%%%%%%%%%%%%%%%%%%%%%%%
%%%%%%%%%%%%%%%%%%%%%%%%%%

Proposition~\ref{P21xx} is the special case of Proposition~\ref{P21} with $\alpha=1/2$  and $x$-independent vital rates~$m\in C^{4,\rho}\big(\R\times J,(0,\infty)\big)$ and $b\in C^{4,0}\big(\R\times J ,(0,\infty)\big)$, noticing that in this case
$$
\Pi_0(a,0)=\exp\left(-\int_0^a m(0,\sigma)\,\rd \sigma\right)\, \Pi_*(a,0)\,,
$$
where $v(a,\cdot)=\Pi_*(a,0)v_0$, $a\in J$,  is for given $v_0\in L_q(\Omega)$ the unique strong solution to the heat equation
$$
\partial_a v=\mathrm{div}_x\big(d(a,x)\nabla_xu\big)\,,\quad (a,x)\in J\times \Omega\,,\qquad v(0,x)=v_0(x)\,,\quad x\in\Omega\,,
$$
subject to Dirichlet boundary conditions if $\delta=0$ or Neumann boundary conditions if $\delta=1$.\\

It is worth pointing out for the case of $x$-independent vital rates~$m$ and $b$ and Neumann boundary conditions $\mathcal{B}=\partial_\nu$ (i.e. $\delta=1$) that the constant function $\mathbf{1}:=[x\mapsto 1]$ belongs to $W_{q,N}^2(\Omega)$ (with subscript~$N$ referring to the Neumann boundary conditions) and satisfies $\Pi_*(a,0)\mathbf{1}=\mathbf{1}$. Therefore,
$$
Q_0(0)\mathbf{1}=  \int_0^{a_m} b(0,a)\, e^{-\int_0^a m(0,s)\rd s}\,\rd a\,\mathbf{1}
$$
so that $\mathbf{1}$ is a positive eigenfunction of the strongly positive compact operator $Q_0(0)$.  Krein-Rutman's theorem (e.g., see \cite[Theorem 12.3]{DanersKochMedina}) implies  
$$
r(Q_0(0))= \int_0^{a_m} b(0,a)\, e^{-\int_0^a m(0,s)\rd s}\,\rd a\,.
$$
Consequently, we obtain from Proposition~\ref{P21}:

\begin{cor}\label{CC}
Assume \eqref{j} with  $x$-independent vital rates~$m=m(\bar u,a)$ and $b=b(\bar u,a)$ and $\delta=1$ (case of Neumann boundary conditions). Set
$$
r_0:=\int_0^{a_m} b(0,a)\, e^{-\int_0^a m(0,s)\rd s}\,\rd a\,.
$$
\item[\bf (a)] If $r_0<1$, then the trivial equilibrium $\phi=0$ to~\eqref{P99} is exponentially asymptotically stable in~$L_1\big(J,W_{q,N}^{2\alpha}(\Omega)\big)$.

\item[\bf (b)] If $r_0>1$, then the trivial equilibrium $\phi=0$ to~\eqref{P99} is unstable in~$L_1\big(J,W_{q,N}^{2\alpha}(\Omega)\big)$.
\end{cor}

In particular, if the death rate dominates the birth rate in the sense that
$$
b(0,a)\le m(0,a)\,,\quad a\in J\,,
$$
then
$$
r_0=\int_0^{a_m} b(0,a)\, e^{-\int_0^a m(0,s)\rd s}\,\rd a \le \int_0^{a_m} m(0,a)\, e^{-\int_0^a m(0,s)\rd s}\,\rd a=1-e^{-\int_0^{a_m} m(0,s)\rd s} <1\,.
$$
Hence, the trivial equilibrium is stable.\\

We also remark the following result on \textit{global} stability of the trivial solution in the special case of Corollary~\ref{CC}~{\bf (a)}. It is the analogue to the non-diffusive case from \cite[Theorem~4]{Pruess83}.

\begin{cor}\label{CCC}
Assume \eqref{j} with $\delta=1$ (case of Neumann boundary conditions) and assume that there are $ b_*\in C(J,(0,\infty))$ and $m_*\in C^\rho(J,\R^+)$ such that
\begin{subequations}
\begin{align}\label{k12}
b(r,a,x)\le b_*(a)\,,\quad m(r,a)\ge m_*(a)\,,\quad (r,a,x)\in \R\times J\times \bar\Omega\,,
\end{align}
and
\begin{align}\label{k2}
\int_0^{a_m} b_*(a)\, e^{-\int_0^a m_*(s)\rd s}\,\rd a<1\,.
\end{align}
Moreover, assume that there is $C_*>0$ such that
\begin{align}\label{k3}
\vert \partial_1 b(r,a,x)\vert +\vert m(r,a,x)\vert+\vert \partial_1 m(r,a,x)\vert\le C_*\,,\quad (r,a)\in\R\times J\times\bar\Omega\,.
\end{align}
\end{subequations}
Then the maximal solution $u(\cdot; u_0)$ to~\eqref{P99} exists globally for each $u_0\in L_1(J,W_{q}^{1}(\Omega))$ with $u_0\ge 0$ and \mbox{$u(t;u_0)\to 0$} in $L_1(J,L_q(\Omega))$ as $t\to\infty$.
\end{cor}

\begin{proof}
Note that
$$
A_\ell(a,x)w:=\mathrm{div}_x\big(d(a,x)\nabla_xw\big)-m_*(a)w\, ,\quad w\in W_{q,N}^2(\Omega)\,,\quad (a,x)\in J\times\Omega\,,
$$
and $b_\ell:=b_*$
satisfy \eqref{A1l} and \eqref{A2l}. We then denote by $(\mS(t))_{t\ge 0}$ the corresponding positive semigroup  on $\E_0=L_1(J,L_q(\Omega))$ defined in \eqref{100} for these $A_\ell$, $b_\ell$ (see Theorem~\ref{IUMJT1}). As in the proof of Corollary~\ref{CC}, supposition~\eqref{k2} implies that the semigroup $(\mS(t))_{t\ge 0}$ has a negative growth bound $\omega_0<0$ (see Proposition~\ref{IUMJprop}).

Let $u_0\in \E_{1/2}=L_1(J,W_{q}^{1}(\Omega))$ with $u_0\ge 0$ be arbitrary and $u(\cdot;u_0)\in C(I(u_0),\E_{1/2})$ be the maximal, positive solution to problem~\eqref{P99} (guaranteed by Proposition~\ref{T1JDE}). We simply write $u=u(\cdot;u_0)$  and note that $u$ satisfies (in a mild sense)
	\begin{align*}
		\partial_t u+\partial_a u&=A_\ell(a) + f(t,a,x)\ , && t\in I(u_0)\, , &  a\in (0,a_m)\, ,& & x\in\Om\, ,\\
		u(t,0,x)&=\int_0^{a_m} b_\ell(a)u(t,a,x)\,\rd a +h(t,x)\, ,& & t\in I(u_0)\, , & & & x\in\Om\, ,\\
		\mathcal{B} u(t,a,x)&=0\ ,& & t\in I(u_0)\, , &  a\in (0,a_m)\, ,& & x\in\partial\Om\, ,
	\end{align*}
where we have introduced the negative functions $f\in C(I(u_0),\E_0)$ and $h\in C(I(u_0),E_0)$ by
	\begin{align*} 
f(t,a,x)&:=\big(m_*(a)-m(\bar u(t,x),a,x)\big)u(t,a,x)\le 0\,,\\ 
h(t,x)&:=\int_0^{a_m} \big(b(\bar u(t,x),a,x)-b_\ell(a)\big)\, u(t,a,x)\,\rd a \le 0 \,.
	\end{align*}
Here, in a mild sense means that $u$ satisfies
\begin{equation}\label{ui}
u(t)=\mS(t)u_0+\int_0^t \mS(t-s) f(s)\,\rd s+ W_{0,0}^{0,h}(t,\cdot)\,,\quad t\in I(u_0)\, ,
\end{equation}
 according to \cite[Corollary~5.8]{WalkerZehetbauerJDE}, where $W_{0,0}^{0,h} \in C(I(u_0),\E_0)$ stems from Lemma~\ref{LemmaF}. Since $f(t,\cdot,\cdot)\le 0$ and $W_{0,0}^{0,h}(t,\cdot)\le 0$ for $t\in I(u_0)$ due to~\eqref{k12} and~\eqref{49998}, it follows from~\eqref{ui}  and the positivity of $(\mS(t))_{t\ge 0}$ that
$$
0\le u(t)\le \mS(t)u_0\,,\quad t\in I(u_0)\,,
$$
in the Banach lattice $\E_0=L_1(J,L_q(\Omega))$. Therefore,
\begin{equation}\label{uii}
\| u(t)\|_{\E_0}\le \|\mS(t)\|_{\ml(\E_0)}\,\|u_0\|_{\E_0}\le c_0\, e^{\omega_0 t}\,\|u_0\|_{\E_0}\,,\quad t\in I(u_0)\,.
\end{equation}
Assumption~\eqref{k3} ensures that there is a constant $C_1>0$ such that
\begin{equation}\label{ui2i}
\|f(t)\|_{L_1(J,W_{q}^{1}(\Omega))} + \|h(t)\|_{W_{q}^{1}(\Omega)}\le C_1\big(1+\|u (t)\|_{L_1(J,W_{q}^{1}(\Omega))}\big)\,,\quad t\in I(u_0)\,,
\end{equation}
while \eqref{E3} yields for every $T>0$ a constant $c(T)>0$ such that
\begin{equation}\label{ui22i}
\|\mS(t)\|_{\ml(L_1(J,W_{q}^{1}(\Omega)))}\le c(T)\,,\quad t\in [0,T]\,.
\end{equation}
It then readily follows from \eqref{ui},~\eqref{ui2i},~\eqref{ui22i}, Lemma~\ref{LemmaF}, and Gronwall's inequality that
$$
\|u(t)\|_{L_1(J,W_{q}^{1}(\Omega))}\le c_1(T)\,,\quad t\in I(u_0)\cap [0,T]\,,
$$
for every $T>0$. Proposition~\ref{T1JDE} now implies that the solution $u$ exists globally, i.e. \mbox{$I(u_0)=[0,\infty)$}. Consequently, we may let $t\to \infty$ in~\eqref{uii} and use $\omega_0<0$ to conclude that $u(t;u_0)\to 0$ in the phase space~$\E_0=L_1(J,L_q(\Omega))$.
\end{proof}

Corollary~\ref{CC} is not  restricted to the particular case of Neumann boundary conditions (just  replace the left-hand side of~\eqref{k2} by the corresponding spectral radius).

%%%%%%%%%%%%%%%%%%%%%%%%%%%%%%%%%%%%%%%%%%%%%%%%%%%%
%%%%%%%%%%%%%%%%%%%%%%%%%%%%%%%%%%%%%%%%%%%%%%%%%%%%
\subsection*{An Instability Result}
%%%%%%%%%%%%%%%%%%%%%%%%%%%%%%%%%%%%%%%%%%%%%%%%%%%%
%%%%%%%%%%%%%%%%%%%%%%%%%%%%%%%%%%%%%%%%%%%%%%%%%%%%

We provide the analogue to \cite[Theorem~6]{Pruess83}:

\begin{prop}\label{P50}
Assume~\eqref{j}. Consider a positive equilibrium 
$$
\phi\in L_1\big(J,W_{q,\mathcal{B}}^{2}(\Omega)\big) \cap C(J,W_{q,\mathcal{B}}^{2\alpha}(\Omega))\,,\quad \phi\ge 0\,,
$$
to \eqref{P99} for which 
\begin{equation}\label{pos}
\partial_1 b(\bar\phi(x),a,x)\ge 0\,,\quad \partial_1 m(\bar\phi(x),a,x)\le 0\,,\qquad (a,x)\in J\times\Omega\,.
\end{equation}
If $s(\G_\phi)\not=0$, then $\phi$ is unstable in~$L_1\big(J,W_{q,\mathcal{B}}^{2\alpha}(\Omega)\big)$.
\end{prop}

\begin{proof}
{\bf (i)} Observe that \eqref{j2b}, \eqref{j3}, \eqref{pos}, and the positivity of $\phi$  entail the strict positivity $b_\phi> 0$ and that $\B_\phi$ is a positive operator on $\E_0=L_1(J,L_q(\Omega))$. Moreover, the maximum principle ensures that $A_\phi(a)$ is resolvent positive for $a\in J$. Theorem~\ref{T:NormCont} now implies that  $\G_\phi=\A_\phi+\B_\phi$ is resolvent positive. 
We then claim that for its spectral bound we have $s(\G_\phi)>-\infty$. Indeed, since $\B_\phi\ge 0$, it follows from \cite[Proposition~12.11]{BFR} that $s(\G_\phi)\ge s(\A_\phi)$. Next, the strict positivity $b_\phi> 0$ and \cite[Corollary~13.6]{DanersKochMedina}  imply that $b_\phi(a)\Pi_\phi(a,0)$ is strongly positive on $L_q(\Omega)$ for $a\in J$. Thus, for $\lambda\in \R$, the operator 
\begin{align*}
Q_{\lambda}(\phi)=\int_0^{a_m}e^{-\lambda a}\,b_\phi(a)\, \Pi_{\phi}(a,0)\,\rd a
\end{align*}
is compact and strongly positive on $L_q(\Omega)$. As in the previous section we infer from \cite[Lemma~3.1]{WalkerIUMJ} that the mapping $[\lambda\mapsto r(Q_{\lambda}(\phi))]$ is continuous and strictly decreasing on $\R$ with 
$$
\lim_{\lambda\to-\infty}r(Q_{\lambda}(\phi))=\infty\,,\qquad \lim_{\lambda\to\infty}r(Q_{\lambda}(\phi))=0\,,
$$
and then from  \cite[Proposition~3.2]{WalkerIUMJ} that $s(\A_\phi)=\lambda_0$ with $\lambda_0\in\R$ being the unique real number such that $r(Q_{\lambda_0}(\phi))=1$.
Thus $s(\G_\phi)\ge \lambda_0$ and $s(\G_\phi)$ is an eigenvalue of $\G_\phi$.\\

{\bf (ii)} Let now $\lambda>0$ be large enough, i.e. $\lambda>\max\{s(\G_\phi),0\}$ with $r(Q_{\lambda}(\phi))<1$. Set $$v:=(\lambda-\G_\phi)^{-1}\phi$$ and note that $v\ge 0$ since $\G_\phi$ is resolvent positive and $\lambda>s(\G_\phi)$, see~\cite[Remark~12.12~(b)]{BFR}. Then Theorem~\ref{IUMJT1}~{\bf (c)} entails that $v$ satisfies (in the sense of mild solutions)
$$
\partial_a v=(-\lambda+A_\phi(a))v +\B_\phi v+\phi\,,\quad a\in J\,,\qquad v(0)=\int_0^{a_m}b_\phi(a)v(a)\,\rd a\,.
$$
Thus, since $\B_\phi v\ge 0$, we deduce 
\begin{align}\label{v1}
v(a)\ge e^{-\lambda a}\,\Pi_\phi(a,0)\, v(0)+\int_0^a e^{-\lambda (a-\sigma)}\,\Pi_\phi(a,\sigma)\,\phi(\sigma)\,\rd \sigma\,,\quad a\in J\,,
\end{align}
and, when plugging this into the initial condition,
\begin{align}
v(0)&\ge \int_0^{a_m}b_\phi(a)\,e^{-\lambda a}\,\Pi_\phi(a,0)\,\rd a \, v(0)+\int_0^{a_m}b_\phi(a)\int_0^a e^{-\lambda (a-\sigma)}\,\Pi_\phi(a,\sigma)\,\phi(\sigma)\,\rd \sigma\,\rd a\nonumber\\
&\ge \int_0^{a_m}b(\bar \phi,a)\,e^{-\lambda a}\,\Pi_\phi(a,0)\,\rd a \, v(0)+\int_0^{a_m}b(\bar \phi,a)\int_0^a e^{-\lambda (a-\sigma)}\,\Pi_\phi(a,\sigma)\,\phi(\sigma)\,\rd \sigma\,\rd a\,, \label{v2}
\end{align}
where we used that $b_\phi(a)\ge b(\bar \phi,a)$ due to~\eqref{pos}. Now, the  equilibrium $\phi$ satisfies
$$
\phi(a)=\Pi_\phi(a,0)\phi(0)\,,\quad a\in J\,,\qquad \phi(0)=Q_0(\phi)\phi(0)\,.
$$
Therefore, using the evolution property
$$
\Pi_\phi(a,\sigma)\Pi_\phi(\sigma,0)=\Pi_\phi(a,0)\,,\quad 0\le \sigma\le a\le a_m\,,
$$
we infer that
\begin{align}\label{p1}
\int_0^a e^{-\lambda (a-\sigma)}\,\Pi_\phi(a,\sigma)\,\phi(\sigma)\,\rd \sigma=\frac{1}{\lambda}\big(1-e^{-\lambda a}\big)\,\Pi_\phi(a,0)\,\phi(0)
\end{align}
and thus
\begin{align}\label{p2}
\int_0^{a_m}b(\phi,a)\int_0^a e^{-\lambda (a-\sigma)}\,\Pi_\phi(a,\sigma)\,\phi(\sigma)\,\rd \sigma\,\rd a=\frac{1}{\lambda}\big(1-Q_\lambda(\phi)\big)\phi(0)\,.
\end{align}
From \eqref{v2} and \eqref{p2} it then follows that
$$
\big(1-Q_\lambda(\phi)\big) v(0)\ge \frac{1}{\lambda}\big(1-Q_\lambda(\phi)\big)\phi(0)
$$
and thus, since $\big(1-Q_\lambda(\phi)\big)^{-1}\ge 0$ as $r(Q_{\lambda}(\phi))<1$, we conclude that 
$\lambda v(0)\ge \phi(0)$.
Using this along with \eqref{p1} in \eqref{v1} we derive
$\lambda v(a)\ge\phi(a)$ for $ a\in J$.
By definition of $v$, this means that
$$
\lambda \big(\lambda-\G_\phi\big)^{-1}\phi\ge \phi\,,\quad \lambda \gg 0\,.
$$
Invoking the exponential representation of the semigroup we obtain
$$
e^{t\G_\phi}\phi=\lim_{n\to\infty}\left(1-\frac{t}{n}\G_\phi\right)^{-n}\phi\ge \phi\,,\quad t\ge 0\,,
$$
from which $\|e^{t\G_\phi}\|_{\ml(\E_0)}\ge 1$ for $t\ge 0$, since $\E_0=L_1(J,L_q(\Omega))$ is a Banach lattice. This implies that $s(\G_\phi)=\omega_0(\G_\phi)\ge 0$. By supposition, we then even have $s(\G_\phi)> 0$. Since $s(\G_\phi)$ is an eigenvalue of $\G_\phi$, we conclude from Theorem~\ref{TInStable} that $\phi$ is unstable.
\end{proof}

A simple consequence is:

\begin{cor}\label{C51}
Assume~\eqref{j}. Consider a positive equilibrium $$\phi\in L_1\big(J,W_{q,\mathcal{B}}^{2}(\Omega)\big) \cap C(J,W_{q,N}^{2\alpha}(\Omega))\,,\quad \phi\ge 0\,,$$
to \eqref{P99} for which \eqref{pos} holds.
If \mbox{$r(Q_0(\phi))>1$}, then $\phi$ is unstable in~$L_1\big(J,W_{q,N}^{2\alpha}(\Omega)\big)$.
\end{cor}

\begin{proof}
It has been observed in the proof of Proposition~\ref{P50} that $s(\G_\phi)\ge \lambda_0$, where $\lambda_0\in\R$ is the uniquely determined from the condition $r(Q_{\lambda_0}(\phi))=1$.
Since $[\lambda\mapsto r(Q_{\lambda}(\phi))]$ is strictly decreasing on $\R$ and $r(Q_0(\phi))>1$, we necessarily have $\lambda_0>0$. 
\end{proof}

%{\tred We then use that $\G_\phi=\A_\phi+\B_\phi$ has compact resolvent according to %Corollary~\ref{Cor:PertCompResolv} and is resolvent positive to derive from the Krein-Rutman theorem %\cite[Corollary 12.14]{BFR} that there exists $\Psi_0\in\mathrm{dom}(\G_\phi)=\mathrm{dom}(\A_\phi)$, %$\Psi_0>0$ such that  
%\begin{align*}
%\G_\phi\Psi_0=\lambda_*\Psi_0\,.
%\end{align*}
%We have seen in \eqref{eigenvals} that the latter implies that
%\begin{align} \label{eqre}
%\left(1-Q_{\lambda_*}(\phi)\right)\Psi_0(0)&= -\int_0^{a_m}b_\phi(a)K_{\phi,\lambda_*}(a)\,\rd %a\,\bar\Psi_0 \ge 0\,,
%%where we used that $-K_{\phi,\lambda_*}\ge 0$ due to \eqref{KL} and assumption~\eqref{pos} along with %$\Psi_0>0$. Consequently, the Krein-Rutman theorem \cite[Corollary~12.4]{DanersKochMedina} implies %$r(Q_{\lambda_*}(\phi))\le 1$.
%}

%%%%%%%%%%%%%%%%%%%%%%%%%%%%%%%%%%%%%%%%%%%%%%%%%%%%%%%%%%%%%%%%%%%%%%
%%%%%%%%%%%%%%%%%%%%%%%%%%%%%%%%%%%%%%%%%%%%%%%%%%%%%%%%%%%%%%%%%%%%%%
%%%%%%%%%%%%%%%%%%%%%%%%%%%%%%%%%%%%%%%%%%%%%%%%%%%%%%%%%%%%%%%%%%%%%%
\begin{appendix}
%%%%%%%%%%%%%%%%%%%%%%%%%%%%%%%%%%%%%%%%%%%%%%%%%%%%%%%%%%%%%%%%%%%%%%
%%%%%%%%%%%%%%%%%%%%%%%%%%%%%%%%%%%%%%%%%%%%%%%%%%%%%%%%%%%%%%%%%%%%%%
%%%%%%%%%%%%%%%%%%%%%%%%%%%%%%%%%%%%%%%%%%%%%%%%%%%%%%%%%%%%%%%%%%%%%%

%%%%%%%%%%%%%%%%%%%%%%%%%%%%%%%%%%%%%%%%%%%%%%%%%%%%%%%%%%%%%%%%%%%%%%
%%%%%%%%%%%%%%%%%%%%%%%%%%%%%%%%%%%%%%%%%%%%%%%%%%%%%%%%%%%%%%%%%%%%%%
%%%%%%%%%%%%%%%%%%%%%%%%%%%%%%%%%%%%%%%%%%%%%%%%%%%%%%%%%%%%%%%%%%%%%%
\section{Proof of Proposition~\ref{Prop:NormCont}}\label{App:NormCont}
%%%%%%%%%%%%%%%%%%%%%%%%%%%%%%%%%%%%%%%%%%%%%%%%%%%%%%%%%%%%%%%%%%%%%%
%%%%%%%%%%%%%%%%%%%%%%%%%%%%%%%%%%%%%%%%%%%%%%%%%%%%%%%%%%%%%%%%%%%%%%
%%%%%%%%%%%%%%%%%%%%%%%%%%%%%%%%%%%%%%%%%%%%%%%%%%%%%%%%%%%%%%%%%%%%%%

\nequation
\aequation

We provide here the proof of Proposition~\ref{Prop:NormCont} which is fundamental for Theorem~\ref{T:NormCont}. We thus impose \eqref{A1l}, \eqref{A2l}, \eqref{Bpart0}, and recall that we consider nonlocal perturbations 
\begin{align*} 
[\B\zeta](a):=\int_0^{a_m} q(a,\sigma)\,\zeta(\sigma)\,\rd \sigma\,,\quad a\in J\,,\quad \zeta\in \E_0\,,
\end{align*}
for some  $q(a,\sigma)=q(a)(\sigma)$ satisfying
\begin{align*}
q\in C\big(J,L_\infty(J,\ml(E_0))\big)\,.%\cap  L_1\big(J,L_\infty(J,\ml(E_\alpha,E_0))\big)
\end{align*}
Then $\B\in\ml(\E_0)$  with
$$
\|\B\|_{\ml(\E_0)}\le a_m\,\|q\|_{\infty}\,,\qquad \|q\|_{\infty}:=\|q\|_{C(J,L_\infty(J,\ml(E_0)))}\,.
$$
For the birth rate we recall that
$b_\ell\in C\big(J,\ml(E_0)\big)$.
We begin with an auxiliary result:

\begin{lem}\label{P51}
Suppose \eqref{A1l}, \eqref{A2l}, and \eqref{Bpart0}. Given $\psi\in\E_0$, let $\mathsf{B}_\psi\in C(\R^+,E_0)$ be defined as in~\eqref{500}. Then, given $T>0$, there is $c_{\mathsf{B}}(T)>0$ such that
\begin{equation}\label{B101}
   \|\mathsf{B}_\psi(\tau)\|_{E_\theta}\le c_{\mathsf{B}}(T)\,\tau^{-\theta}\,\|\psi\|_{\E_0}\,,\quad \tau\in (0,T]\,,\quad \theta\in\{0,\vartheta\}\,.
    \end{equation}
Moreover, given $\ve>0$ and $\kappa\in (0,a_m/2)$, there is $\delta:=\delta(T,\kappa,\ve)>0$ such that
\begin{equation*}\label{B102}
\|\mathsf{B}_{\B\zeta}(\tau_1)-\mathsf{B}_{\B\zeta}(\tau_2)\|_{E_0}\le \ve \|\zeta\|_{\E_0}\,,\qquad \tau_1,\tau_2\in [\kappa,T]\ \text{ with }\  \vert \tau_1-\tau_2\vert \le \delta\,,
    \end{equation*}
whenever $\zeta\in\E_0$.
\end{lem}

\begin{proof}
Estimate~\eqref{B101} follows for $\theta=0$ from the fact that $[\psi\mapsto \mathsf{B}_\psi]\in \ml \big(\E_0, C(\R^+,E_0)\big)$, see~\eqref{BBBB}. It is derived from Gronwall's inequality as shown for both cases $\theta=0,\vartheta$ in~\cite[Formula~(2.2)]{WalkerIUMJ}. 

In order to prove the continuity property of $\mathsf{B}_{\B\zeta}$
for $\zeta\in\E_0$, set $\psi:=\B\zeta\in\E_0$. We write 
\begin{align}
    \mathsf{B}_\psi(\tau)&=\int_{(\tau-a_m)_+}^\tau  b_\ell(\tau-a)\, \Pi_\ell(\tau-a,0)\, \mathsf{B}_\psi(a)\, \rd a\nonumber\\
&\qquad +\int_0^{(a_m-\tau)_+}  b_\ell(\tau+a)\, \Pi_\ell(\tau+a,a)\, \psi(a)\, \rd a\nonumber\\
&=:\mathsf{B}_\psi^1(\tau)+\mathsf{B}_\psi^2(\tau)\label{p1x}
    \end{align}
for~$\tau\ge 0$ and note from \eqref{EOx} that
\begin{align}\label{c1}
c_1:=\max_{0\le\sigma\le a\le a_m}\|\Pi_\ell(a,\sigma)\|_{\ml(E_0)}<\infty\,.
   \end{align} 
Choose $\delta_1:=\delta_1(T,\ve,\kappa)\in(0,\min\{a_m/2,\kappa\})$  such that
\begin{equation}\label{p4}
   2\,c_1\,c_{\mathsf{B}}(T)\,\|b\|_{L_\infty(J,\ml(E_0))}\, \max\{T,1\}\,\delta_1\le \frac{\ve}{4\|\B\|_{\ml(\E_0)}}\,.    
\end{equation}
Due to Lemma~\ref{Knu} we may choose $\delta_2:=\delta_2(T,\ve,\kappa)>0$ such that
\begin{equation}\label{p5}
   T\,c_{\mathsf{B}}(T)\,\|b\|_{L_\infty(J,\ml(E_0))}\, \|\Pi_\ell(s_1,0)-\Pi_\ell(s_2,0)\|_{\ml(E_0)}  \le \frac{\ve}{4\|\B\|_{\ml(\E_0)}}  \end{equation}
for $s_1, s_2\in [\delta_1,T\wedge a_m]$ with $\vert s_1-s_2\vert \le \delta_2$ and
\begin{equation}\label{p6}
  \|b_\ell(\bar s_1)-b_\ell(\bar s_2)\|_{\ml(E_0)}\, c_1\, \max\{c_{\mathsf{B}}(T)T,1\}\le \frac{\ve}{4\|\B\|_{\ml(\E_0)}}   
\end{equation}
for $\bar s_1, \bar s_2\in [0,a_m]$ with $\vert \bar s_1-\bar s_2\vert \le \delta_2$. Set $$\delta_0:=\delta_0(T,\ve,\kappa):=\min\{\delta_1,\delta_2\}>0\,.$$ Then, for $\kappa\le \tau_2\le \tau_1\le T$ with $\vert \tau_1-\tau_2\vert \le \delta_0$ we note that $(\tau_1-a_m)_+<\tau_2-\delta_1$ and obtain from \eqref{B101}-\eqref{p6} 
\begin{align*}
\| \mathsf{B}_\psi^1&(\tau_1)-\mathsf{B}_\psi^1(\tau_2)\|_{E_0}\\
& \le
\int_{(\tau_2-a_m)_+}^{(\tau_1-a_m)_+} \| b_\ell(\tau_2-a)\|_{\ml(E_0)}\,\| \Pi_\ell(\tau_2-a,0)\|_{\ml(E_0)}\,\| \mathsf{B}_\psi(a)\|_{E_0}\, \rd  a \\
&\quad +\int_{(\tau_1-a_m)_+}^{\tau_2} \| b_\ell(\tau_1-a)- b_\ell(\tau_2-a)\|_{\ml(E_0)}\,\| \Pi_\ell(\tau_1-a,0)\|_{\ml(E_0)}\,\| \mathsf{B}_\psi(a)\|_{E_0}\, \rd  a \\
&\quad +\int_{(\tau_1-a_m)_+}^{\tau_2-\delta_1} \| b_\ell(\tau_2-a)\|_{\ml(E_0)}\,\| \Pi_\ell(\tau_1-a,0)-\Pi_\ell(\tau_2-a,0)\|_{\ml(E_0)}\,\| \mathsf{B}_\psi(a)\|_{E_0}\, \rd  a 
\\
&\quad + \int_{\tau_2-\delta_1}^{\tau_2} \| b_\ell(\tau_2-a)\|_{\ml(E_0)}\,\| \Pi_\ell(\tau_1-a,0)-\Pi_\ell(\tau_2-a,0)\|_{\ml(E_0)}\,\| \mathsf{B}_\psi(a)\|_{E_0}\, \rd  a 
\\
&\quad +\int_{\tau_2}^{\tau_1} \| b_\ell(\tau_1-a)\|_{\ml(E_0)}\,\| \Pi_\ell(\tau_1-a,0)\|_{\ml(E_0)}\,\| \mathsf{B}_\psi(a)\|_{E_0}\, \rd  a \\
& \le \|b_\ell\|_{L_\infty(J,\ml(E_0))}\, c_1 \, c_{\mathsf{B}}(T)\, \|\psi\|_{\E_0}\, \vert(\tau_1-a_m)_+-(\tau_2-a_m)_+ \vert \\
&\quad +\sup_{\substack{s_1,s_2\in [0, a_m]\\ \vert s_1-s_2\vert\le \delta_0}}\| b_\ell(s_1)- b_\ell(s_2)\|_{\ml(E_0)}\,c_1\,c_{\mathsf{B}}(T)\,T\, \|\psi\|_{\E_0}  \\
&\quad +\|b_\ell\|_{L_\infty(J,\ml(E_0))}\,   \sup_{\substack{s_1,s_2\in [\delta_1, a_m]\\\vert s_1-s_2\vert\le \delta_0}} \| \Pi_\ell(s_1,0)-\Pi_\ell(s_2,0)\|_{\ml(E_0)}\,c_{\mathsf{B}}(T)\,T\, \|\psi\|_{\E_0}
\\
&\quad + 2\, \|b_\ell\|_{L_\infty(J,\ml(E_0))}\, c_1\,\delta_1\, c_{\mathsf{B}}(T)\,T\, \|\psi\|_{\E_0}
 +\|b_\ell\|_{L_\infty(J,\ml(E_0))}\, c_1 \, c_{\mathsf{B}}(T)\, \|\psi\|_{\E_0}\, \vert\tau_1-\tau_2 \vert \\
&\le \frac{\ve}{\|\B\|_{\ml(\E_0)}}\,\|\psi\|_{\E_0}\,.
\end{align*}
Consequently,  using $\psi=\B\zeta$ with $\B\in\ml(\E_0)$ we derive
\begin{align}\label{psi22}
\| \mathsf{B}_\psi^1&(\tau_1)-\mathsf{B}_\psi^1(\tau_2)\|_{E_0} \le  \ve\,\|\zeta\|_{\E_0}\,,\qquad \kappa\le \tau_2\le \tau_1\le T\,,\quad \vert \tau_1-\tau_2\vert \le \delta_0\,.
\end{align}
For $\mathsf{B}_\psi^2$ we use Lemma~\ref{Knu} to find
 \mbox{$\eta:=\eta(\ve,\kappa)>0$} such that
$$
\|b_\ell\|_{L_\infty(J,\ml(E_0))}\,\|\Pi_\ell(a+\tau_1,a)-\Pi_\ell(a+\tau_2,a)\|_{\ml(E_0)}\le \frac{\ve}{4\|\B\|_{\ml(\E_0)}}\,,
$$
whenever $\kappa\le\tau_2\le \tau_1\le T$, $ a\in[0,(a_m-\tau_1)_+]$, $\vert \tau_1-\tau_2\vert\le \eta$. Let $\delta_3>0$ with
$$
\| q \|_{\infty}\,\delta_3\le \frac{\ve}{4}
$$
and set $$\delta:=\delta(T,\ve,\kappa):=\min\{\delta_0,\eta,\delta_3\}\,.$$ Then we obtain for $\kappa\le \tau_2\le \tau_1\le T$ with $\vert \tau_1-\tau_2\vert \le \delta$
that
\begin{align*}
\| \mathsf{B}_\psi^2&(\tau_1)-\mathsf{B}_\psi^2(\tau_2)\|_{E_0}\\
& \le
\int_0^{(a_m-\tau_1)_+}\| b_\ell(a+\tau_1)-b_\ell(a+\tau_2)\|_{\ml(E_0)}\,\| \Pi_\ell(a+\tau_1,a)\|_{\ml(E_0)}\,\| \psi(a)\|_{E_0}\, \rd  a\\
&\quad + \int_0^{(a_m-\tau_1)_+}\| b_\ell(a+\tau_2)\|_{\ml(E_0)}\,\| \Pi_\ell(a+\tau_1,a)-\Pi_\ell(a+\tau_2,a)\|_{\ml(E_0)}\,\| \psi(a)\|_{E_0}\, \rd  a \\
&\quad + \int_{(a_m-\tau_1)_+}^{(a_m-\tau_2)_+}\| b_\ell(a+\tau_2)\|_{\ml(E_0)}\,\| \Pi_\ell(a+\tau_2,a)\|_{\ml(E_0)}\,\| \psi(a)\|_{E_0}\, \rd  a \\
&\le  \sup_{\substack{s_1,s_2\in [0, a_m]\\ \vert s_1-s_2\vert\le \delta}}\| b_\ell(s_1)- b_\ell(s_2)\|_{\ml(E_0)}\, c_1\,\|\psi\|_{\E_0}\\
&\quad +\|b_\ell\|_{L_\infty(J,\ml(E_0))}\,   \sup_{a\in [0, (a_m-\tau_1)_+]} \| \Pi_\ell(a+\tau_1,a)-\Pi_\ell(a+\tau_2,a)\|_{\ml(E_0)}\,\|\psi\|_{\E_0}\\
&\quad +\|b_\ell\|_{L_\infty(J,\ml(E_0))}\, c_1 \int_{(a_m-\tau_1)_+}^{(a_m-\tau_2)_+} \| \psi(a)\|_{E_0}\, \rd  a \,.
\end{align*}
Since $\psi=\B\zeta$ we get from~\eqref{Bpart0} that
\begin{align*}
\int_{(a_m-\tau_1)_+}^{(a_m-\tau_2)_+} \| \psi(a)\|_{E_0}\, \rd  a&\le \int_{(a_m-\tau_1)_+}^{(a_m-\tau_2)_+} \int_0^{a_m}\| q(a,\sigma)\zeta(\sigma)\|_{E_0}\,\rd \sigma \rd  a\nonumber\\
&\le \|\zeta\|_{\E_0} \, \| q \|_{\infty}\,\vert (a_m-\tau_2)_+-(a_m-\tau_1)_+\vert\,.
\end{align*}
Gathering the previous computations and using~\eqref{p6} and $\psi=\B\zeta$ with $\B\in\ml(\E_0)$, we deduce that
\begin{align}\label{psi223}
\| \mathsf{B}_\psi^2&(\tau_1)-\mathsf{B}_\psi^2(\tau_2)\|_{E_0} \le  \ve\,\|\zeta\|_{\E_0} \,.
\end{align}
for $\kappa \le \tau_2\le \tau_1\le T$ with $\vert \tau_1-\tau_2\vert \le \delta$.
Consequently,  Lemma~\ref{P51} follows from \eqref{p1x},~\eqref{psi22}, and~\eqref{psi223}.
\end{proof}

%%%%%%%%%%%%%%%%%%%%%%%%%%%%%%%%%%%
%%%%%%%%%%%%%%%%%%%%%%%%%%%%%%%%%%%
\subsection*{Proof of Proposition~\ref{Prop:NormCont}}
%%%%%%%%%%%%%%%%%%%%%%%%%%%%%%%%%%%
%%%%%%%%%%%%%%%%%%%%%%%%%%%%%%%%%%%

For the proof of Proposition~\ref{Prop:NormCont} we suppose \eqref{A1l}, \eqref{A2l}, and~\eqref{Bpart0}. We have to show that $\mathcal{V}\mS(t)\in \mathcal{K}(\E_0)$ for each $t>0$, where
\begin{align*} 
\mathcal{V}\mS(t)\zeta:=\int_0^t\mS(t-s)\,\B\,\mS(s)\zeta\,\rd s\,,\quad \zeta\in\E_0\,.
\end{align*}
To this end we use Simon's criterion for compactness in $\E_0=L_1(J,E_0)$ (see \cite[Theorem~1]{Simon_87}). Note first from~\eqref{E3} that, for $t>0$ and~$\zeta\in\E_0$,
\begin{align*}\ 
\|\mathcal{V}\mS(t)\zeta\|_{\E_0}&\le \int_0^t\|\mS(t-s)\|_{\ml(\E_0)}\,\|\B\|_{\ml(\E_0)}\,\|\mS(s)\|_{\ml(\E_0)}\,\|\zeta\|_{\E_0}\,\rd s\\
&\le M_0^2\,\|\B\|_{\ml(\E_0)}\, e^{ \varkappa_0  t}\,t\, \|\zeta\|_{\E_0}\,,
\end{align*}
so that $\mathcal{V}\mS:(0,\infty)\to \ml(\E_0)$.

 Let $t\in (0,T)$ be fixed and consider a sequence $(\phi_j)_{j\in\N}$ with $\|\phi_j\|_{\E_0}\le k_0$ for $j\in\N$. Then $(\mathcal{V}\mS(t)\phi_j)_{j\in\N}$ is bounded in $\E_0$ as just shown.\\

\noindent {\bf (a)} We introduce 
$$
\psi_s^j:=\B\,\mS(s)\phi_j\,,\quad s>0\,,\quad j\in\N\,.
$$ 
Note that $[s\mapsto \psi_s^j]\in C\big((0,\infty),\E_0\big)$ with, recalling~\eqref{Bpart0},
\begin{equation*}
\begin{split}
\|\psi_s^j(a)\|_{E_0}&\le \|q(a,\cdot)\|_{L_\infty(J,\ml(E_0))}\,\|\mS(s)\phi_j\|_{\E_0}\,,\quad a\in J\,,\quad s\in (0,T]\,,\quad j\in\N\,,
\end{split}
\end{equation*}
Hence, invoking~\eqref{E3},
\begin{equation}\label{psis2}
\begin{split}
\|\psi_s^j(a)\|_{E_0}&\le 
c_0(T)\,k_0\,,\quad a\in J\,,\quad s\in (0,T]\,,\quad j\in\N\,,
\end{split}
\end{equation}
where
$c_0(T):=\|q\|_{\infty}M_0\,e^{\vert\varkappa_0\vert T}$,
and then
\begin{equation}\label{psis}
\|\psi_s^j\|_{\E_0}\le a_m\, c_0(T)\,k_0\,,\quad s\in (0,T]\,,\quad j\in\N\,.
\end{equation} 
Together with~\eqref{B101}  this yields
\begin{equation}\label{psis3}
\|\mathsf{B}_{\psi_s^j}(\tau)\|_{E_\theta}\le c_\mathsf{B}(T)\, a_m\, c_0(T)\,k_0\, \tau^{-\theta}\,,\quad s,\tau\in (0,T]\,,\quad j\in\N\,,\quad \theta\in \{0,\vartheta\}\,.
\end{equation}
Let $0<h<\min\{a_m/2,t\}$. Then, due to~\eqref{100} we have
\begin{align}
&\int_0^{a_m-h}\left\|{[\mathcal{V}\mS(t)\phi_j]}(a+h)-[\mathcal{V}\mS(t)\phi_j](a)\right\|_{E_0}\,\rd a\nonumber\\
&\le \int_0^{a_m-h}\int_0^t\left\|[\mS(t-s)\psi_s^j](a+h)-[\mS(t-s)\psi_s^j](a)\right\|_{E_0}\,\rd s\, \rd a\nonumber\\
&\le \int_0^{a_m-h}\int_{(t-a)_+}^t\|\Pi_\ell(a+h,a+h-t+s)-\Pi_\ell(a,a-t+s)\|_{\ml(E_0)}\,\nonumber\\
&\qquad\qquad\qquad\qquad\qquad\qquad\qquad\qquad\qquad\qquad\times \|\psi_s^j(a+h-t+s)\|_{E_0}\, \rd s\, \rd a \nonumber\\
&\quad +\int_0^{a_m-h}\int_{(t-a)_+}^t\|\Pi_\ell(a,a-t+s)\|_{\ml(E_0)}\, \|\psi_s^j(a+h-t+s)-\psi_s^j(a-t+s)\|_{E_0}\,\rd s\, \rd a\nonumber\\
&\quad +\int_0^{a_m-h}\int_{(t-a-h)_+}^{(t-a)_+}\big\|\Pi_\ell(a+h,a+h-t+s)\psi_s^j(a+h-t+s) \nonumber\\
&\qquad\qquad\qquad\qquad\qquad\qquad\qquad\qquad\qquad\qquad -\Pi_\ell(a,0)\mathsf{B}_{\psi_s^j}(t-s-a)\big\|_{E_0}\,\rd s\, \rd a\nonumber\\
&\quad +\int_0^{a_m-h}\int_0^{(t-a-h)_+}\big\|\Pi_\ell(a+h,0)\mathsf{B}_{\psi_s^j}(t-s-a-h)\nonumber\\
&\qquad\qquad\qquad\qquad\qquad\qquad\qquad\qquad\qquad\qquad -\Pi_\ell(a,0)\mathsf{B}_{\psi_s^j}(t-s-a)\big\|_{E_0}\,\rd s\, \rd a\nonumber\\
&=: I+II+III+IV\,.\label{I}
\end{align}
We then treat  each integral separately. Let $\ve>0$ be arbitrary in the following.\vspace{2mm}

{\bf (i)} Choose $\kappa:=\kappa(\ve,T,t)\in (0,\min\{a_m/2,t\})$ such that
$$
2 c_1\, c_0(T)\,k_0\, (T+a_m)\,\kappa \le \frac{\ve}{6}
$$
and  $\eta_1:=\eta_1(T,\ve,\kappa)\in (0,a_m)$ such that (see Lemma~\ref{Knu})
$$
\sup_{\substack{(a_1,\sigma_1) ,(a_2,\sigma_2)\in S_\kappa\\ \vert(a_1,\sigma_1) -(a_2,\sigma_2)\vert\le \eta_1 }}\|\Pi_\ell(a_1,\sigma_1) -\Pi_\ell(a_2,\sigma_2)\|_{\ml(E_0)} \,  c_0(T)\,k_0\,a_m\,T  \le \frac{\ve}{6}\,.
$$
Then, from  \eqref{psis2}, \eqref{c1} we have, for $2h<\eta_1$,
\begin{align*}
 I
&\le c_0(T)k_0 \int_0^{\kappa}\int_{(t-a)_+}^t\|\Pi_\ell(a+h,a+h-t+s)-\Pi_\ell(a,a-t+s)\|_{\ml(E_0)}\,  \rd s\, \rd a \\
&\ + c_0(T)k_0 \int_{\kappa}^{a_m-h}\int_{(t-a)_+}^{t-\kappa}\|\Pi_\ell(a+h,a+h-t+s)-\Pi_\ell(a,a-t+s)\|_{\ml(E_0)}\,  \rd s\, \rd a \\
&\ + c_0(T)k_0 \int_{\kappa}^{a_m-h}\int_{t-\kappa}^t\|\Pi_\ell(a+h,a+h-t+s)-\Pi_\ell(a,a-t+s)\|_{\ml(E_0)}\,  \rd s\, \rd a \\
&\le 2 c_1 c_0(T)k_0(T+a_m)\kappa\\
&\ + c_0(T)k_0 \int_{\kappa}^{a_m-h}\int_{(t-a)_+}^{t-\kappa}\|\Pi_\ell(a+h,a+h-t+s)-\Pi_\ell(a,a-t+s)\|_{\ml(E_0)}\,  \rd s\, \rd a \\
&\le 2 c_1c_0(T)k_0 (T+a_m)\,\kappa\\
&\  +c_0(T)\,k_0\,a_m\,T\,\sup_{\substack{(a_1,\sigma_1) ,(a_2,\sigma_2)\in S_\kappa\\ \vert(a_1,\sigma_1) -(a_2,\sigma_2)\vert\le \eta_1 }}\|\Pi_\ell(a_1,\sigma_1) -\Pi_\ell(a_2,\sigma_2)\|_{\ml(E_0)}   
\end{align*}
and therefore
\begin{align}\label{n2}
 I &\le \frac{\ve}{3}\,,\quad 2h<\eta_1\,.
\end{align}

{\bf (ii)} We choose $\eta_2:=\eta_2(T,\ve)>0$ according to~\eqref{Bpart0} such that
$$
c_1\sup_{\substack{\tau_1, \tau_2\in J\\ \vert\tau_1-\tau_1\vert\le\eta_2}} \left\|q(\tau_1,\cdot)-  q(\tau_2,\cdot) \right\|_{L_\infty(J,\ml(E_0))} a_m \,T\, M_0\,e^{\vert\varkappa_0\vert T} \,  k_0\le \frac{\ve}{3} \,.
$$
We then use~\eqref{c1}, recall $\psi_s^j=\B\,\mS(s)\phi_j$, and invoke~\eqref{Bpart0} and \eqref{E3} to get, for $h<\eta_2$,
\begin{align*}
II&\le c_1\int_0^{a_m-h}\int_{(t-a)_+}^t\int_0^{a_m}\left\|\big(q(a+h-t+s,\sigma)-  q(a-t+s,\sigma)\big)[\mS(s)\phi_j](\sigma) \right\|_{E_0}\, \rd\sigma\rd s \rd a\nonumber\\
&\le c_1\int_0^{a_m-h}\int_{(t-a)_+}^t\left\|q(a+h-t+s,\cdot)-  q(a-t+s,\cdot) \right\|_{L_\infty(J,\ml(E_0))}\,  \|\mS(s)\phi_j\|_{\E_0}\, \rd s\, \rd a\nonumber\\
&\le c_1\sup_{\substack{\tau_1, \tau_2\in J\\ \vert\tau_1-\tau_1\vert\le\eta_2}} \left\|q(\tau_1,\cdot)-  q(\tau_2,\cdot) \right\|_{L_\infty(J,\ml(E_0))} a_m \,T\, M_0\,e^{\vert\varkappa_0\vert T} \,  k_0
\end{align*}
and therefore
\begin{align}\label{n3}
 II &\le \frac{\ve}{3}\,,\quad h<\eta_2\,.
\end{align}
%\vspace{2mm}

{\bf (iii)} It follows from \eqref{c1}, \eqref{psis2}, and \eqref{psis3} that
\begin{align}\label{n4}
III&\le c_1 \, c_0(T)\,k_0\, a_m\big(1+c_\mathsf{B}(T)\, a_m\big)\, h\,.
\end{align}

{\bf (iv)} Finally, we choose $\bar\kappa:=\bar\kappa(\ve,T)\in (0,a_m/2)$ such that
$$
2 c_1\,a_m^2\, c_\mathsf{B}(T)\, c_0(T)\,k_0\,\bar\kappa \le \frac{\ve}{6}
$$
and invoke then \eqref{psis2} and Lemma~\ref{P51} to find $\eta_3:=\eta_3(T,\ve)>0$ such that
\begin{align*}%\label{320}
\|\mathsf{B}_{\psi_s^j}(\tau_1)-\mathsf{B}_{\psi_s^j}(\tau_2)\|_{E_0}\le  \frac{\ve}{6c_1a_mT}\,,\qquad \tau_1,\tau_2\in [\bar \kappa,T]\ \text{ with }\  \vert \tau_1-\tau_2\vert \le \eta_3\,.
    \end{align*}
Let $C(\vartheta)>0$ be the constant from \eqref{II.Equation(5.3.8)} (for $\Pi$ replaced by $\Pi_\ell$). Using the previous estimate along with \eqref{c1} and~\eqref{psis3} we then derive, for $h<\eta_3$,
\begin{align*}
 IV&\le \int_0^{a_m-h}\int_0^{(t-a-h-\bar\kappa)_+}\left\|\Pi_\ell(a+h,0)-\Pi_\ell(a,0)\right\|_{\ml(E_\vartheta E_0)} \| \mathsf{B}_{\psi_s^j}(t-s-a-h)\|_{E_\vartheta}\,\rd s\, \rd a\\
&\quad +\int_0^{a_m-h}\int_0^{(t-a-h-\bar\kappa)_+}\|\Pi_\ell(a,0)\|_{\ml(E_0)}\,\|\mathsf{B}_{\psi_s^j}(t-s-a-h)-\mathsf{B}_{\psi_s^j}(t-s-a)\|_{E_0}\,\rd s\, \rd a\\
&\quad +\int_0^{a_m-h}\int_{(t-a-h-\bar\kappa)_+}^{(t-a-h)_+}\Big\{ \|\Pi_\ell(a+h,0)\|_{\ml(E_0)}\,\|\mathsf{B}_{\psi_s^j}(t-s-a-h)\|_{E_0}\\
&\qquad\qquad\qquad\qquad\qquad\qquad\qquad\qquad+\|\Pi_\ell(a,0)\|_{\ml(E_0)}\,\|\mathsf{B}_{\psi_s^j}(t-s-a)\|_{E_0}\Big\}\,\rd s\, \rd a\\
&\le C(\vartheta)\, h^\vartheta\, c_\mathsf{B}(T)\, a_m\, c_0(T)\,k_0\, \int_0^{a_m-h}\int_0^{(t-a-h-\bar\kappa)_+} (t-s-a-h)^{-\vartheta}\,\rd s\, \rd a\\
&\quad + c_1\, a_m\, T\,  \frac{\ve}{6c_1a_mT} + 2 c_1\,a_m^2\, c_\mathsf{B}(T)\,  c_0(T)\,k_0\,\bar\kappa
\end{align*}
and therefore
\begin{align}\label{n5}
 IV &\le \frac{\ve}{3}+C(\vartheta)\, c_\mathsf{B}(T)\, a_m\, c_0(T)\,k_0\, a_m\,T^{1-\vartheta}\, h^\vartheta\,,\quad h<\eta_3\,.
\end{align}

{\bf (v)} Consequently, we conclude from \eqref{I}-\eqref{n5} that
$$
\varlimsup_{h\to 0}\,\sup_{j\in\N}\int_0^{a_m-h}\left\|{[\mathcal{V}\mS(t)\phi_j]}(a+h)-[\mathcal{V}\mS(t)\phi_j](a)\right\|_{E_0}\,\rd a\le \ve
$$
and thus, since $\ve>0$ was arbitrary,
\begin{align}\label{n6}
\lim_{h\to 0}\,\sup_{j\in\N}\int_0^{a_m-h}\left\|{[\mathcal{V}\mS(t)\phi_j]}(a+h)-[\mathcal{V}\mS(t)\phi_j](a)\right\|_{E_0}\,\rd a=0\,.
\end{align}
\noindent {\bf (b)} Finally, for $j\in\N$ we have from \eqref{EOx}, \eqref{psis3}, and \eqref{psis2} that
\begin{align*}
\int_0^{a_m}\|[\mathcal{V}\mS(t)\phi_j](a)\|_{E_\vartheta}\,\rd a&\le \int_0^{a_m}\int_0^{(t-a)_+}\|\Pi(a,0)\|_{\ml(E_0,E_\vartheta)}\,\|\mathsf{B}_{\psi_s^j}(t-s-a)\|_{E_0}\,\rd s\, \rd a\\
&\  +\int_0^{a_m}\int_{(t-a)_+}^t\|\Pi(a,a-t+s)\|_{\ml(E_0,E_\vartheta)}\,\|\psi_s^j(a-t+s)\|_{E_0}\,\rd s\, \rd a\\
&\le  M_\vartheta \big(c_\mathsf{B}(T)\, a_m+1\big) c_0(T)\,k_0 \int_0^{a_m} e^{\varpi a}\,a^{-\vartheta}\,\rd a\,.
\end{align*}
Since the right-hand side is finite and due to the compact embedding of $E_\vartheta$ in $E_0$ we conclude that
\begin{align}\label{n7}
\left\{\int_{a_1}^{a_2}[\mathcal{V}\mS(t)\phi_j](a)\,\rd a\,;\, j\in \N\right\}\ \text{ is relatively compact in $E_0$ for  $0<a_1<a_2<a_m$}\,.
\end{align}
We now infer from  \eqref{n6}, \eqref{n7}, and \cite[Theorem~1]{Simon_87}  that the sequence $(\mathcal{V}\mS(t)\phi_j)_{j\in\N}$ is relatively compact in~$\E_0=L_1(J,E_0)$. Since $t>0$ was arbitrary, this yields Proposition~\ref{Prop:NormCont}.
\qed

%%%%%%%%%%%%%%%%%%%%%%%%%%%%%%%%%%%%%%%%%%%%%%%%%%%%%%%%%%%%%%%%%%%%%%
%%%%%%%%%%%%%%%%%%%%%%%%%%%%%%%%%%%%%%%%%%%%%%%%%%%%%%%%%%%%%%%%%%%%%%
%%%%%%%%%%%%%%%%%%%%%%%%%%%%%%%%%%%%%%%%%%%%%%%%%%%%%%%%%%%%%%%%%%%%%%
\section{Parabolic Evolution Operators}\label{App:EvolSys}
%%%%%%%%%%%%%%%%%%%%%%%%%%%%%%%%%%%%%%%%%%%%%%%%%%%%%%%%%%%%%%%%%%%%%%
%%%%%%%%%%%%%%%%%%%%%%%%%%%%%%%%%%%%%%%%%%%%%%%%%%%%%%%%%%%%%%%%%%%%%%
%%%%%%%%%%%%%%%%%%%%%%%%%%%%%%%%%%%%%%%%%%%%%%%%%%%%%%%%%%%%%%%%%%%%%%

\nequation
\bequation

Parabolic evolution operators are thoroughly treated in~\cite{LQPP} to which we refer. We only provide here their most important properties that we have used in the previous sections. 

\subsection*{Basic Definition} Let $E_1\hookrightarrow E_0$ be a densely injected Banach couple, $J=[0,a_m]$, and
$$
J_\Delta:= \{(a,\sigma)\in J\times J\,;\,  0\le\sigma\le a\}\,,\quad J_\Delta^*:= \{(a,\sigma)\in J\times J\,;\,  0\le\sigma< a\}\,.
$$
We consider $$
A:J\to \mathcal{A}(E_0)
$$
with $\dom(A(a))=E_1$ for each $a\in J$, 
where $\mathcal{A}(E_0)$ means the closed linear operators in~$E_0$. \\

Following \cite[Section~II.2.1]{LQPP} we say that {\bf $A$ generates a parabolic evolution operator~$\Pi$  on~$E_0$ with regularity subspace $E_1$}, provided that $\Pi:J_\Delta\to \ml(E_0)$ is such that
\begin{subequations}
\begin{equation}\label{II.Equation(2.1.2)}
\Pi\in C\big(J_\Delta,\ml_s(E_0)\big)\cap C\big(J_\Delta^*,\ml(E_0,E_1)\big)
\end{equation}
satisfying
\begin{equation}
\Pi(a,a)=1_{E_0}\,,\qquad \Pi(a,\sigma)=\Pi(a,s)\Pi(s,\sigma)\,,\quad (a,s), (s,\sigma)\in J_\Delta\,,
\end{equation}
and, for $a\in J$,
\begin{equation}\label{II.Equation(2.1.6)}
\Pi(\cdot,a)\in C^1\big((a,a_m),\ml(E_0)\big) \,,\quad \Pi(a,\cdot)\in C^1\big([0,a),\ml_s(E_1,E_0)\big) 
\end{equation}
with
\begin{equation}
\partial_1 \Pi(a,\sigma)=A(a)\Pi(a,\sigma)\,,\quad \partial_2 \Pi(a,\sigma)=- \Pi(a,\sigma) A(\sigma)\,,\quad (a,\sigma)\in J_\Delta^*\,.
\end{equation}
\end{subequations}

In the following, let
\begin{equation*}
A\in  C^\rho\big(J,\mathcal{H}(E_1,E_0)\big)
\end{equation*}
be fixed with $\rho>0$. Then~\cite[II.~Corollary~4.4.2]{LQPP} ensures that $A$ generates a unique parabolic evolution operator $\Pi$  on $E_0$ with regularity subspace~$E_1$ in the above sense.

\subsection*{Basic Estimates} Given an interpolation space $E_\theta=(E_0,E_1)_\theta$ with $\theta\in [0,1]$,  there are $\varpi\in\R$ and $M_\theta\ge 1$ such that
	\begin{equation}\label{EOx}
	\|\Pi(a,\sigma)\|_{\ml(E_\theta)}+(a-\sigma)^\theta\,\|\Pi(a,\sigma)\|_{\ml(E_0,E_\theta)}\le M_\theta e^{\varpi (a-\sigma)}\,,\qquad (a,\sigma)\in J_\Delta \,,
	\end{equation}
according to~\cite[II.~Lemma~5.1.3]{LQPP}.

\subsection*{Solvability of Cauchy Problems} For $x\in E_0$ and \mbox{$f\in L_1(J,E_0)$}, 
the {\it mild solution}  $v\in C(J,E_0)$ to the Cauchy problem
\begin{subequations}
\begin{equation}\label{CP}
\partial_a v=A(a)v+ f(a)\,,\quad a\in \dot{J}:=(0,a_m]\,,\qquad v(0)= x\,,
\end{equation}
is given by
\begin{equation}\label{VdKx}
v(a)=\Pi(a,0)x+\int_0^a\Pi(a,\sigma)\,f(\sigma)\,\rd \sigma\,,\quad a\in J\,.
\end{equation} 
\end{subequations}
If $x\in E_\vartheta$ for some $\vartheta\in [0,1]$ and $f\in C^\theta(J,E_0)+C(J,E_\theta)$ with $\theta \in (0,1]$ (with admissible interpolation functors), then $$v\in  C(J,E_\vartheta)\cap C^1(\dot{J},E_0)\cap C(\dot{J},E_1)$$ is a strong solution to~\eqref{CP}. Actually, if, in addition, $ x\in E_1$, then $$v\in C^1(J,E_0)\cap C(J,E_1)\,.$$ See \cite[II.~Theorem~1.2.1,~Theorem~1.2.2]{LQPP}.

\subsection*{Continuity Properties } Given $\vartheta\in (0,1)$, there is $C(\vartheta)>0$ such that
\begin{equation}\label{II.Equation(5.3.8)}
\|\Pi (a+h,0)-\Pi  (a,0)\|_{\ml(E_\vartheta,E_0)}\le C(\vartheta) h^{\vartheta}\,,\quad 0\le a\le a+h\le a_m\,,
\end{equation}
according to \cite[II.~Equation~(5.3.8)]{LQPP}. Moreover:

\begin{lem}\label{Knu}
For $\ve>0$ and $\kappa\in (0,a_m)$ given, there is $\eta:=\eta(\ve,\kappa)>0$ such that
$$
\|\Pi (a_1,\sigma_1)-\Pi (a_2,\sigma_2)\|_{\ml(E_0)}\le \ve\,,\quad (a_1,\sigma_1), (a_2,\sigma_2)\in S_\kappa\,,\quad \vert (a_1,\sigma_1)-(a_2,\sigma_2)\vert\le \eta\,,
$$
where
$
S_\kappa:=\{(a,\sigma)\in \Delta_J^*\,;\,  \kappa\le  a-\sigma\}.
$
\end{lem}

\begin{proof}
This follows from the fact that $\Pi\in C\big(\Delta_J^*,\ml(E_0)\big)$ is uniformly continuous on the compact subset
$S_\kappa$
of $\Delta_J^*$.
\end{proof}

\subsection*{Positivity} If $E_0$ is an ordered Banach space and
$A(a)$ is resolvent positive\footnote{An operator $A\in \mathcal{A}(E_0)$ is {\it resolvent positive}, if there is $\lambda_0\ge 0$ such that $(\lambda_0,\infty)\subset \rho(A)$ and $$(\lambda-A)^{-1}\in \ml_+(E_0)\,,\quad \lambda>\lambda_0\,.$$} for each $a\in J$,
then \cite[II.~Theorem~6.4.1,~Theorem~6.4.2]{LQPP} imply that the evolution operator $\Pi$ is positive, that is, $\Pi(a,\sigma)\in \ml_+(E_0)$ for each $(a,\sigma)\in J_\Delta$.

\end{appendix}

\section*{Declarations}

\subsection*{Ethical Approval}
 The submitted work is original and has not been published elsewhere in any form or language (partially or in full).

\subsection*{Competing interests} Not applicable.

\subsection*{Funding} Not applicable.

\subsection*{Availability of data and materials} Not applicable.

%%%%%%%%%%%%%%%%
\bibliographystyle{siam}
\bibliography{AgeDiff_230303}
%%%%%%%%%%%%%%%%

\end{document}